%% file: Main_NonlinearPOD.tex
\documentclass[11pt]{article}
\usepackage{mathrsfs}
\newenvironment{proof}[1][Proof]{\begin{trivlist}
\item[\hskip \labelsep {\bfseries #1}]}{\end{trivlist}}
\newcounter{example}
  
\newtheorem{lemma}{Lemma}[section]
\newtheorem{theorem}{Theorem}[section]

\newtheorem{proposition}{Proposition}[section]

\usepackage{latexsym}
\usepackage{amsmath}
\usepackage{amssymb}
\usepackage{graphicx}
\usepackage{textcomp}
\usepackage{multirow}
\usepackage{enumerate}
\usepackage{cancel}
\usepackage{caption}  
\usepackage[usenames,dvipsnames]{color}
\usepackage[left, pagewise]{lineno}
\usepackage[ruled,vlined]{algorithm2e}
\usepackage{hyperref}
\hypersetup{
    colorlinks,%
    citecolor=black,%
    filecolor=black,%
    linkcolor=black,%
    urlcolor=black
}
\usepackage{wrapfig}
\usepackage{authblk}
\graphicspath{{./Figures/}}
\makeatletter
      \newcommand\figcaption{\def\@captype{figure}\caption}
      \newcommand\tabcaption{\def\@captype{table}\caption}
\makeatother
\usepackage{fancyhdr, lastpage}
\title{Nonlinear Model Reduction Based on the Finite Element Method With Interpolated Coefficients: Semilinear Parabolic Equations}
\author[]{Zhu Wang \thanks{Email: wangzhu@ima.umn.edu. URL: \url{http://www.ima.umn.edu/\textasciitilde wangzhu/} }}
\affil[]{Institute for Mathematics and its Applications, University of Minnesota, \\ 
354 Lind Hall, Minneapolis, MN 55455-0134, U. S. A.}
\begin{document}
\maketitle
\begin{abstract}
For nonlinear reduced-order models, especially for those with high-order polynomial nonlinearities or non-polynomial nonlinearities, the computational complexity still depends on the dimension of the original dynamical system. 
As a result, the reduced-order model loses its computational efficiency, which, however, is its the most significant advantage. 
Nonlinear dimensional reduction methods, such as the discrete empirical interpolation method, have been widely used to evaluate the nonlinear terms at a low cost. 
But when the finite element method is utilized for the spatial discretization, nonlinear snapshot generation requires inner products to be fulfilled, which costs lots of off-line time. 
Numerical integrations are also needed over elements sharing the selected interpolation points during the simulation, which keeps on-line time high. 
To overcome these issues and develop an efficient finite element discretization algorithm, in this paper, we extend the finite element method with interpolated coefficients, also known as the group finite element method or the product approximation, to nonlinear reduced-order models. 
The proposed approach approximates the nonlinear function in the reduced-order model by its finite element interpolation, 
which makes coefficient matrices of the nonlinear terms pre-computable and, thus, leads to great savings in the computational efforts. 
Due to the separation of spatial and temporal variables in the finite element interpolation, 
the discrete empirical interpolation method can be directly applied on the nonlinear functions in the same manner as that in the finite difference setting. 
Therefore, the main computational hurdles when applying the discrete empirical interpolation method in the finite element context are conquered. 
We also establish a rigorous asymptotic error estimation, which shows that the proposed approach achieves the same accuracy as that of the standard finite element method under certain smoothness assumptions of the nonlinear functions. 
Several numerical tests are presented to validate the proposed method and verify the theoretical results. 
\end{abstract}



\input{notation_old}
\input{Introduction}
\input{POD}
\input{FEIC}
\input{Analysis}

\input{Numerical}

\input{FEIC_DEIM}

\input{Conclusion}
\section*{Acknowledgments}
The author is grateful to Dr. Saifon Chaturantabut for sharing the Matlab implementation codes of the discrete empirical interpolation method. 
\bibliographystyle{plain}
\bibliography{./../../../D_bibliography}
\end{document}

%% file: notation_old.tex
\newcommand{\lp}{\left(}
\newcommand{\rp}{\right)}
\newcommand{\lno}{\left\|}
\newcommand{\rno}{\right\|}
\newcommand{\id}{\mathrm{d}}

\newcommand{\cP}{\mathcal{P}}
\newcommand{\mN}{\mathcal{N}}
\newcommand{\oracP}{\overrightarrow{\cP}}
\newcommand{\ou}{\overline{u}}
\newcommand{\opsi}{\overline{\psi}}
\newcommand{\oq}{\overline{q}}
\newcommand{\oT}{\overline{T}}
\newcommand{\E}{\mathbbm{E}}
\newcommand{\orho}{\overline{\rho}}

\newcommand{\s}{\sigma}
\renewcommand{\k}{\kappa}
\newcommand{\p}{\partial}
\newcommand{\om}{\omega}
\newcommand{\Om}{\Omega}
\newcommand{\pOm}{\partial \Omega}
\newcommand{\e}{\epsilon}
\renewcommand{\a}{\alpha}
\renewcommand{\b}{\beta}
   \newcommand{\eps}{\varepsilon}
   \newcommand{\EX}{{\Bbb{E}}}
   \newcommand{\PX}{{\Bbb{P}}}

\newcommand{\nd}{{\nabla \cdot}}

\newcommand{\cF}{{\cal F}}
\newcommand{\cD}{{\cal D}}
\newcommand{\cO}{{\cal O}}

\newtheorem{assumption}{Assumption}[section]
\newtheorem{remark}{Remark}[section]
\newcommand{\kmodel}{k_{\mbox{Model}}}
\newcommand{\obu}{\overline{\bf u}}
\newcommand{\oobu}{\overline{\overline{\bf u}}}
\newcommand{\bP}{{\bf P}}
\newcommand{\bu}{{\bf u}}
\newcommand{\bk}{{\bf k}}
\newcommand{\bs}{{\bf s}}
\newcommand{\be}{{\bf e}}
\newcommand{\oou}{\overline{\overline{u}}}
\newcommand{\op}{\overline{p}}
\newcommand{\of}{\overline{f}}
\newcommand{\obf}{\overline{\bf f}}
\newcommand{\ow}{\overline{w}}
\newcommand{\ov}{\overline{v}}
\newcommand{\ophi}{\overline{\phi}}
\newcommand{\oS}{\overline{S}}
\newcommand{\obS}{\overline{\bf S}}
\newcommand{\bv}{{\bf v}}
\newcommand{\obv}{\overline{\bf v}}
\newcommand{\bc}{{\bf c}}
\newcommand{\by}{{\bf y}}
\newcommand{\bw}{{\bf w}}
\newcommand{\bW}{{\bf W}}
\newcommand{\bU}{{\bf U}}
\newcommand{\obw}{\overline{\bf w}}
\newcommand{\bz}{{\bf z}}
\newcommand{\bZ}{{\bf Z}}
\newcommand{\obZ}{\overline{\bf Z}}
\newcommand{\bff}{{\bf f}}
\newcommand{\bee}{{\bf e}}
\newcommand{\bn}{{\bf n}}
\newcommand{\bx}{{\bf x}}
\newcommand{\bX}{{\bf X}}
\newcommand{\bH}{{\bf H}}
\newcommand{\bV}{{\bf V}}
\newcommand{\bL}{{\bf L}}
\newcommand{\bg}{{\bf g}}
\newcommand{\bj}{{\bf j}}
\newcommand{\br}{{\bf r}}
\newcommand{\grads}{\nabla^s}
\def\PP{{{\rm l}\kern - .15em {\rm P} }}
\def\PN2{{\PP_{N}-\PP_{N-2}}}
\newcommand{\erf}[1]{\mbox{erf}\left(#1\right)}
\newcommand{\D}{\nabla}
\newcommand{\I}{\mathbb{I}}
\newcommand{\N}{\mathbb{N}}
\newcommand{\R}{\mathbb{R}}
\newcommand{\Z}{\mathbb{Z}}
\newcommand{\mathR}{\R}
\newcommand{\mathN}{\N}
\newcommand{\mathZ}{\Z}
\newcommand{\mathI}{\mathbbm{I}}
\newcommand{\btau}{\boldsymbol{\tau}}
\newcommand{\bphi}{\boldsymbol{\varphi}}
\newcommand{\bvarphi}{\boldsymbol{\varphi}}
\newcommand{\bpsi}{\boldsymbol{\psi}}
\newcommand{\bfeta}{\boldsymbol{\eta}}
\newcommand{\blambda}{\boldsymbol{\lambda}}
\newcommand{\bPhi}{\boldsymbol{\Phi}}
\newcommand{\obphi}{\overline {\boldsymbol{\phi}}}
\newcommand{\bomega}{\boldsymbol{\omega}}
\newcommand{\bsigma}{\boldsymbol{\sigma}}
\newcommand{\orhoprime}{\overline{\rho^{\prime}}}
\newcommand{\bus}{{\bf u}^*}
\newcommand{\By}{\mathcal B(\by)}
\newcommand{\eci}[1]{\mathcal E_{#1}}
\newcommand{\dpyi}[1]{\delta_{#1}^+(\by)}
\newcommand{\dmyi}[1]{\delta_{#1}^-(\by)}
\newcommand{\cA}{{\mathcal A(\by)}}
\newcommand{\dyi}[1]{\delta_{#1}(\by)}
\newcommand{\cG}{{\mathcal G(\bx,\by)}}
\newcommand{\cGi}[1]{{\mathcal G_{#1}(\bx,\by)}}
\newcommand{\pti}{\partial_i}
\newcommand{\ptii}[1]{\partial_{#1}}
\newcommand{\ba}{{\bf a}}

\newcommand{\rey}{\mbox{Re}}

\newcommand{\tnp}{t^{k+1}}
\newcommand{\bb}{{\bf b}}
\newcommand{\fnp}{f^{k+1}}
\newcommand{\prp}{P_R^{'}}
\newcommand{\pr}{P_R}
\newcommand{\rn}{r^k}
\newcommand{\ur}{\bu_r}
\newcommand{\urn}{\bu_r^k}
\newcommand{\utn}{\bu_t^k}
\newcommand{\urnp}{\bu_r^{k+1}}
\newcommand{\utnp}{\bu_t^{k+1}}
\newcommand{\un}{\bu^k}
\newcommand{\unp}{\bu^{k+1}}
\newcommand{\vr}{\bv_r}
\newcommand{\wrn}{\bw_r^k}
\newcommand{\etan}{\eta^k}
\newcommand{\etanp}{\eta^{k+1}}
\newcommand{\prn}{\phi_r^k}
\newcommand{\prnp}{\phi_r^{k+1}}
\newcommand{\prz}{\phi_r^{0}}
\newcommand{\prm}{\phi_{r}^{M}}
\newcommand{\sNn}{\sum\limits_{k=0}^{M-1}}
\newcommand{\sN}{\sum\limits_{k=0}^{M}}
\newcommand{\asNn}{\frac{1}{M-1}\sNn}
\newcommand{\asN}{\frac{1}{M}\sN}
\newcommand{\uph}{\upsilon^h}
\newcommand{\half}{\frac{1}{2}}
\newcommand{\wrh}{w_r^h}
\newcommand{\wrkh}{w_{r, k}^h}
\newcommand{\urh}{u_r^h}
\newcommand{\urkh}{u_{r, k}^h}
\newcommand{\urkph}{u_{r, k+1}^h}
\newcommand{\prh}{\phi_r^h}
\newcommand{\prth}{\phi_{r, t}^h}
\newcommand{\prkh}{\phi_{r, k}^h}
\newcommand{\prkph}{\phi_{r, k+1}^h}
\newcommand{\vrh}{v_r^h}
\newcommand{\rrh}{\rho_r^h}

%% file: Introduction.tex
\section{Introduction}
Control and optimization problems in realistic engineering applications often require repeated numerical simulations of large-scale dynamical systems. 
If a fast or real-time control strategy is desired, a brute force direct numerical simulation (DNS) is impractical. 
Therefore, the proper orthogonal decomposition method (POD) has been frequently used to generate a reduced-order model (ROM) that can be utilized as an alternative of the original dynamical system (hereinafter referred to as the full-order model) \cite{Gun03}. 
Such a ROM only contains a handful of degrees of freedom (DOF), yet is computationally feasible and free of storage issues.
The POD method has been successfully applied in many scientific and engineering problems (see, e.g., a brief summarization in \cite{wang2011two}). 
However, for complex dynamical systems, the original promise of the POD as an efficient yet accurate approximation remains to be fulfilled. 
On the one hand, it may achieve erroneous results even when POD basis functions capture most of the system energy \cite{aubry1993preserving}. 
On the other hand, the efficiency is severely limited by the nonlinearity of the system \cite{chaturantabut2010nonlinear}. 
For treating the first issue, research has been done in two main directions: 
(i) to improve the POD basis functions \cite{amsallem2012nonlinear,bergmann2009enablers,bjorklund2012hierarchical,borggaard2006interval,bui2007goal,burkardt2006pod,carlberg2011low,daescu2008dual,iollo2000stability,kunisch2010optimal,Noack2003}; 
(ii) to improve the ROM by modeling the effect of discarded POD basis \cite{borggaard2008reduced,borggaard2011artificial,
cordier2010calibration,couplet2005cro,fang2013nonlinear,galletti2007accurate,Galletti2004,wang2011variational,iollo2000two,kalb2007intrinsic,lumley1996dynamical,noack2011reduced,noack2002low,
noack2005,Pod01,podvin2009proper,PL98,rempfer1996investigations,rempfer2000low,RF94,sirisup2004spectral,ullmann2010pod,wang2011two,wang2011closure,weller2009feedback}. 
To our knowledge, the research for seeking a reliable reduced-order modeling approach for general complex systems is still active, however, beyond the scopes of this paper. 
Instead, we will focus on the second issue and develop an efficient algorithm for the nonlinear ROM.

Indeed, the computational efficiency of POD-ROMs relies on two key components: 
(i) the dimension of the ROM is extremely small; 
(ii) the vectors and matrices in the reduced system can be precomputed. 
For linear systems or nonlinear systems with low order polynomial nonlinearities, both ingredients are satisfied. 
However, for highly nonlinear dynamical systems, the second one does not hold any more 
because matrices associated with the nonlinear terms have to be evaluated and assembled at each time step or iteration. 
Since POD basis functions are global, the evaluation would depend on the dimension of the full-order model. 
The reassembling process would greatly increase the computational cost. 
Therefore, several methods have been proposed to resolve this issue. 
%
%


The two-level algorithms proposed in \cite{wang2011two} are motivated by the observation that only a small number of leading POD basis functions are kept in the ROM, and they have larger length scales than the discarded ones. 
If a computation on a fine mesh is employed to obtain all the POD basis, 
one should be able to use a much coarser mesh to represent the leading POD basis. 
Therefore, in two-level algorithms, nonlinear closure terms are computed on the coarse mesh. 
This way can decrease the computational cost by an order of magnitude, while achieving the same level of accuracy as the simulation on the fine mesh. 
However, the optimal choice of the coarse mesh still needs to be investigated. 

Much work devotes to approximate nonlinear terms at a few selected spatial points or within the neighborhoods around the points.  
The trajectory piecewise-linear method (TPWL), presented in \cite{rewienski2003trajectory,rewienski2006model}, reduces a nonlinear model to a weighted sum of linearized models at selected points along a state trajectory. 
The missing points estimation method (MPE) was developed in \cite{astrid2004reduction,astrid2008missing,willcox2006unsteady}. 
In that approach, the full-order system was first reduced by choosing equations only associated with selected spatial points, restricting the POD basis onto these points, and then projecting the extracted system onto the space spanned by the POD basis. 
The empirical interpolation method (EIM) was first proposed in \cite{barrault2004empirical} for approximating non-affine parameter dependence functions to enable an efficient offline-online computational strategy, and then was further applied to approximate nonlinear functions in \cite{grepl2007erb}. 
The method selects interpolation points by greedy algorithms guided by a posteriori error estimates. 
However, in certain problems, even the most optimal basis set is of large size \cite{eftang2011parameter}, which reduces the efficiency of the algorithm. 
Thus, several improved greedy algorithms have been proposed in \cite{hesthaven2011efficient}. 
The best-points interpolation method (BPIM) was introduced in \cite{nguyen2008best}, which determines interpolation points from a least-square minimization problem. 
The discrete empirical interpolation method (DEIM) was introduced in \cite{chaturantabut2010nonlinear} and analyzed in \cite{chaturantabut2012state}. 
The method combines projection with interpolation and chooses interpolation points from the POD basis of the nonlinear functions. 
Since a certain coefficient matrix can be precomputed when approximating nonlinear functions, the complexity of the  POD-ROM reduces to be proportional to the number of selected spatial indices. 

Another method is the group POD method (GPOD) proposed in \cite{dickinson2010nmr}, which extends the group finite element method to the POD setting. 
In that paper, the authors considered dynamical systems with polynomial nonlinearities, such as the Burgers equations. 
The POD approximation of the quadratic nonlinear term was rewritten in a group format. 
If $r$ POD basis functions are used in the ROM, the GPOD requires $r^3-\frac{1}{2}r^2-\frac{1}{2}r$ flops less than the standard POD implementation in the computation of the nonlinear term. 
However, the approach is limited to polynomial nonlinearities. 

In this report, we focus on developing an efficient finite element (FE) discretization algorithm for nonlinear ROMs. 
In particular, we are interested in applying the DEIM to reduce the intensive computational efforts for evaluating the nonlinear terms. 
However, in the FE setting, there are two major issues that degrade the effectiveness of the DEIM: 
(i) generating nonlinear snapshots, which are to be used for seeking the nonlinear POD basis, requires calculations of the inner product in the nonlinear terms, which costs lots of off-line computation time; 
(ii) the on-line simulation needs evaluations of the inner product over the elements sharing selected DEIM points. 
Repeated numerical integrations will increase the on-line simulation time, especially, in cases such as complex flow simulations when many DEIM points are required to achieve a good approximation. 

To overcome these hurdles, we develop the finite element method with interpolated coefficients (FEIC) for nonlinear POD-ROMs. 
This method is also known as the group finite element method or the product approximation, which has been successfully applied to find numerical solutions to nonlinear partial differential equations  \cite{chen1989error,chen1991approximation,christie1981product,fletcher1983group,douglas1975effect,
larsson1989interpolation,lopez1988stability,
sanz1984interpolation,tourigny1990product,xiong2008convergence}. 
Indeed, we replace the nonlinear function in the POD-ROM with its FE interpolation directly. 
This simple change would lead to great savings in the computation: 
First, the coefficient matrices of the nonlinear terms can be computed beforehand, thus, the evaluation of the nonlinear terms does not involve any numerical quadratures during the on-line simulation; 
Second, for the ROMs discretized by the new approach, the DEIM can be directly applied in the same manner as that in the finite difference context.  
In fact, the nonlinear snapshots become a collection of vectors of nonlinear function values, therefore, 
neither does the nonlinear data generation require any numerical integrations. 
The accuracy of this approach is, of course, restricted to the smoothness of the nonlinear functions. 
However, when the nonlinear functions possess certain smoothness ((H2) and (H3) in Section \ref{sec:analysis}), the FEIC can achieve the same accuracy as the standard FE discretization of the ROM. 
Therefore, the advantages of the new approach over the standard FE discretization of nonlinear POD-ROMs cannot be over-emphasized: 
(i) the FEIC is easier to implement and computationally more efficient; 
(ii) the FEIC achieves the same accuracy when nonlinear functions satisfy smoothness assumptions; 
and
(iii) the FEIC is more suitable for combining with the DEIM and further reduce the computational complexity. 

Note that the GPOD method is based on a similar idea. 
Distinguishing from it, the proposed method in this paper doesn't group any variables in terms of the POD basis. 
Therefore, it fits well for ROMs with polynomial or non-polynomial nonlinearities.  
The rest of this paper is organized as follows: 
a brief introduction to the POD method is presented in Section \ref{sec:pod}; 
the FEIC of POD-ROMs for semilinear parabolic equations is proposed in Section \ref{sec:feic}; 
a rigorous asymptotic error estimate is developed in Section \ref{sec:analysis}; 
several examples are tested in Section \ref{sec:numeric} to numerically demonstrate the accuracy and efficiency of the new approach. 
Wherein, the first two examples are used for validation and the third example is for verification. 
The combination of the proposed approach with the DEIM is discussed in Section \ref{sec:FEIC-DEIM}, whose effectiveness is then illustrated by revisiting the first two examples used in Section \ref{sec:numeric}. 
A few conclusions are drawn in the last section with several ongoing research directions we are pursuing listed. 

%% file: POD.tex
\section{The POD Method}\label{sec:pod}  
In this section, we briefly introduce the (time-continuous) proper orthogonal decomposition method. 
For detailed discussions, the reader is referred to \cite{chapelle2012galerkin,HLB12,KV01,KV02,Singler2013,Sir87a}. 

Let $\cal H$ be a real Hilbert space endowed with inner product $(\cdot, \cdot)_{\cal H}$ and norm $\|\cdot\|_{\cal H}$. 
Assume the data $\mathcal{V}$ (so-called snapshots), which is a collection of time-varying functions $y(x, t) \in L^2(0, T; \cal H)$,  the POD method seeks a low-dimensional basis, $\varphi_1(x), \ldots, \varphi_r(x) \in \cal H$, that optimally approximates the data. 
Mathematically speaking, for any positive $r$, the POD basis is determined by minimizing the error between the data and its projection onto the basis, that is, 
\begin{equation}
\min_{ \{\varphi_j\}_{j=1}^r }
\int_0^T 
  \Big\| y(\cdot, t) - 
  \sum_{j=1}^r \left( y(\cdot, t), \varphi_j(\cdot) \right)_{\cal H} \, \varphi_j(\cdot) 
  \Big\|_{\cal H}^2\, \id t,
\label{pod_min}
\end{equation}
subject to the conditions that $(\varphi_i, \varphi_j)_{\cal H} = \delta_{ij}, \ 1 \leq i, j \leq r$, where $\delta_{ij}$ is the Kronecker delta. 
In order to solve \eqref{pod_min}, one can consider the eigenvalue problem
\begin{equation}
K \, v_j = \lambda_j \, v_j \, ,
\label{pod_eigenvalue}
\end{equation}
where $K$ is a compact linear operator that satisfies 
$K v_j(s) = \int_0^T (y(\cdot, t), y(\cdot, s))_{\cal H}\,v_j(t)\, \id t$,
$v_j$ is the $j$-th eigenvector, and
$\lambda_j\leq \ldots\leq\lambda_2\leq\lambda_1$ are positive eigenvalues. 

It can then be shown that the solution
of \eqref{pod_min} is given by
\begin{equation}
\varphi_{j}(\cdot) = \frac{1}{\sqrt{\lambda_j}} \, \int_0^T v_j(t) \, y(\cdot , t) \, \id t, 
\quad 1 \leq j \leq r.
\label{pod_basis_formula}
\end{equation}

\begin{proposition}(\cite{KV01})
Let the POD basis given by \eqref{pod_basis_formula} be of rank $r$, the POD projection error satisfies
\begin{equation}
\int_0^T
  \Big\| y(\cdot, t) - 
  \sum_{j=1}^r \left( y(\cdot, t), \varphi_j(\cdot) \right)_{\cal H} \, \varphi_j(\cdot) 
  \Big\|_{\cal H}^2 \, \id t
  = \sum\limits_{j>r} \lambda_j.
 \label{eq:poderr} 
\end{equation}
\label{pr:poderr}
\end{proposition}

\begin{remark}
In practice, discrete data is always considered. The POD basis for an ensemble of snapshots, 
$\mathcal{V}=\text{ span }\{y(x, t_1), \ldots, y(x, t_M)\}$, 
is to minimize the projection error 
\begin{equation}
\min_{ \{\varphi_j\}_{j=1}^r }\sum_{\ell=1}^M 
  \Big\| y(\cdot, t_{\ell}) - 
  \sum_{j=1}^r \left( y(\cdot, t_{\ell}), \varphi_j(\cdot) \right)_{\cal H} \, \varphi_j(\cdot) 
  \Big\|_{\cal H}^2.
\label{pod_min_discete}
\end{equation}
The solution can be obtained by solving the eigenvalue problem 
$K \, v_j = \lambda_j \, v_j$ first, 
where $K \in \R^{M \times M}$ is the snapshot correlation matrix with 
$\displaystyle K_{\ell k} = \left( y(\cdot, t_{\ell}), y(\cdot, t_k) \right)_{\cal H}$ and $1\leq \ell, k\leq M$,
then the POD basis is given by $\varphi_{j}(\cdot) = \frac{1}{\sqrt{\lambda_j}} \, \sum_{{\ell}=1}^{M} (v_j)_{\ell} \, y(\cdot , t_{\ell})$, 
$1 \leq j \leq r$.
\end{remark}

\begin{remark} 
Finite element solutions are commonly utilized as the snapshots, that is, the data $y(x, t) = \sum\limits_{\iota=1}^{n_{dof}}{\bf Y}_{\iota}(t) h_{\iota}(x)$, where {\bf Y}(t) is the approximation solution vector at time $t$, $h_{\iota}(x)$ is the $\iota$-th FE nodal basis, and $n_{dof}$ is the number of DOF in the spatial discretization.  
Then the POD basis can be written as  
\begin{equation}
\varphi_j(x)= \sum\limits_{\iota=1}^{n_{dof}}{\bf Q}_{\iota j} h_{\iota}(x),
\label{eq:podbasis}
\end{equation} 
where the coefficient matrix ${\bf Q}$ is with the entry ${\bf Q}_{\iota j}= \frac{1}{\sqrt{\lambda_j}} \, \int_0^T v_j(t) \, Y_{\iota}(t) \, \id t$ for time-continuous data or ${\bf Q}_{\iota j}= \frac{1}{\sqrt{\lambda_j}} \, \sum_{{\ell}=1}^{M} (v_j)_{\ell} \, Y_{\iota}(t_{\ell})$ for time-discrete data. 
For a detailed discussion on implementing the POD method in the FE setting, readers are referred to \cite{KV99}.
\end{remark}

%% file: FEIC.tex
\section{The POD-ROM and Finite Element Discretizations} \label{sec:feic}
Let $\Om$ be a convex domain in $\R^d$ with the smooth boundary $\p \Om$, $d = 1, 2, 3$. 
Also let $\mathcal{T}_h$ be a collection of quasi uniform elements that partitions the domain. 
The elements are line segments if $d=1$, triangles if $d= 2$, and polyhedra if $d= 3$. 
The parameter $h$ is the maximal diameter of the elements.  
Denoted by  $X$ the space $L^2(\Om)$ equipped with the inner product $(\cdot, \cdot)$ and norm $\|\cdot\|$; 
by $V$ the Sobolev space $H^1_0 (\Om)=\{v|v\in H^1(\Om), v|_{\p \Om} =0\}$ with $H^1$ semi-norm $|\cdot|_1$ and $H^1$ norm $\|\cdot\|_1$; 
and by $V^h$ the space of piecewise continuous functions on $\Om$ that reduce to polynomials of degree $\leq m$ on each element of $\mathcal{T}_h$, which satisfies $V^h\subset V$. 
We assume the semilinear problems that we will consider in this section admit a unique solution $u = u(x, t)$ on the time interval $[0, T]$, which ranges in $V$. 

We consider the equivalent variational problems of semilinear parabolic equations with homogeneous boundary conditions: 
To find $u(x, t) \in V$, such that,  either 
\begin{equation}
\left( \frac{\p u}{\p t}, v \right) + a\left(u, v\right) + \left( N(u), v \right) = ( f, v ),\qquad \forall\, v\in V, \label{eq:pdewf1}
\end{equation}  
or
\begin{equation}
\left( \frac{\p u}{\p t}, v \right) + a\left(u, v\right) + \left( N(u), \nabla v \right) = ( f, v ),\qquad \forall\, v\in V, \label{eq:pdewf2}
\end{equation}  
with the initial condition 
\begin{equation*}
u(x, 0) = u_0(x), \qquad \forall\, x\in \Om.
\end{equation*}  
Where $f= f(x,t)$ is a source term independent with $u$, the bilinear form $a(\cdot, \cdot): V\times V \rightarrow \R$ is continuous and coercive, that is, 
there exist constants $\alpha$ and $\beta$ such that
\begin{eqnarray}
|a(u, v)|
&\leq& \alpha \, \| u \|_1 \, \| v \|_1,
\qquad \forall \, u, \, v \in V,
\label{a_continuity} \\
|a(u, u)|
&\geq& \beta \, \| u \|_1^2,
\qquad \forall \, u \in V.
\label{a_coercivity}
\end{eqnarray}
$N(u)$ is a nonlinear function of $u$. 
In particular, we are interested in cases in which $N(u)$ possesses either a non-polynomial nonlinearity or a high order polynomial nonlinearity. 
Due to the similarity between \eqref{eq:pdewf1} and \eqref{eq:pdewf2}, to shorten the presentation, we only discuss \eqref{eq:pdewf1} in the sequel, but will comment a theoretical result of \eqref{eq:pdewf2} in Remark \ref{rem:pdewf2} and test a problem governed by \eqref{eq:pdewf2} in Section \ref{sec:numeric}. 

Many well-established methods can be used for seeking a numerical solution to the equation \eqref{eq:pdewf1}. 
However, when repeated numerical simulations are required, direct simulations result in a huge computational cost and become infeasible. 
Therefore, the POD method has been widely used for generating a reduced-order model.  

\subsection{The POD-ROM}
Given snapshots consisting of either numerical solutions or experimental data, the POD basis functions $\{ \varphi_1(x), \ldots, \varphi_r(x) \}$ are determined by \eqref{pod_min}-\eqref{pod_basis_formula}, where $\mathcal{H}=L^2$ in our considerations. 
The associated POD approximation of $u(x, t)$ is given by
   \begin{equation}
     u_r(x, t) \equiv \sum_{j=1}^r \varphi_j(x) a_j(t),  
     \label{ch1:podapproximation}
   \end{equation}
   where $\{{a}_j(t)\}_{j=1}^r$ are the sought time-varying POD basis coefficient functions.  
Substituting the POD approximation \eqref{ch1:podapproximation} into \eqref{eq:pdewf1}, applying the Galerkin method, and considering the POD basis functions are orthonormal, we obtain the Galerkin projection-based POD-ROM (POD-G-ROM): 
to find $u_r(x, t) \in V_r = \text{span}\left\{\varphi_1, \ldots, \varphi_r\right\}$, such that,  
\begin{equation}
\left( \frac{\p u_r}{\p t}, v_r \right) + a\left(u_r, v_r\right) + \left(N(u_r), v_r \right) = \left( f, v_r \right), \qquad \forall\, v_r \in V_r, \label{eq:podg}
\end{equation}
and
\begin{equation*}
u_r(x, 0) = u_{0, r}(x) \in V_r \qquad \forall\, x\in \Om.
\end{equation*}
Let ${\ba}(t)=[a_1(t), \ldots, a_r(t)]^\intercal$, the POD-G-ROM can be rewritten in terms of POD basis functions as: 
\begin{equation}
\dot{\ba} = {\bf A} + {\bf B} \ba + {\bf C}(\ba), \label{eq:podg_a} \\
\end{equation} 
with, for example,  
\begin{equation*}
\ba(0) = \left(u_0(x), \Phi(x) \right). \nonumber
\end{equation*} 
Where ${\bf A}_k = (f, \varphi_k)$, ${\bf B}_{jk} = -a(\varphi_j, \varphi_k)$, 
$\left({\bf C}(\ba)\right)_{k}= -\left(N(\sum_{j=1}^r \varphi_j a_j(t)) , \varphi_k \right)$, for $k=1, \ldots, r$, and $\Phi = [\varphi_1, \ldots, \varphi_r]^\intercal$. 
The resulted dynamic system \eqref{eq:podg_a} is of dimension $r$, which is much smaller than the number of DOF in the full-order model. 
Once $\ba$ is obtained, the POD approximation solution $u_r$ can be recovered by \eqref{ch1:podapproximation}.

The most significant advantage of the ROM is its computational efficiency. 
Indeed, the matrix ${\bf B}$ can be precomputed. 
In certain cases, we can also compute matrices in ${\bf C}$ beforehand.  
For example, in Navier-Stokes equations, one can write 
$\left({\bf C}(\ba)\right)_{k}= \ba^\intercal {\bf C}^k \ba $, where
$[{\bf C}^k]_{ij} = -\sum_{i=1}^r\sum_{j=1}^r(\varphi_i\cdot\nabla \varphi_j, \varphi_k)$. 
The matrices only need to be computed once and can be used repeatedly in on-line simulations. 
However, this attractive property does not hold when the nonlinear function $N(u)$ is with a higher oder polynomial nonlinearity or a non-polynomial one. 
As a result, the computational efficiency of the ROM decays. 
Especially, when the FE is used for a spatial discretization. 
Therefore, in the rest of this paper, we restrict ourselves to the FE methods of the nonlinear ROMs \eqref{eq:podg} and develop a new efficient FE discretization algorithm.

\subsection{Finite Element Discretizations}  
Let $V_r^h = \text{span}\left\{\varphi_1^h, \ldots, \varphi_r^h\right\}$, where $\varphi_j^h$ is the finite element discretization of $\varphi_j$, $j= 1, \ldots, r$. The standard finite element discretization of the POD-G-ROM \eqref{eq:podg} (POD-FEM) is to find $\urh(x, t) \in V_r^h$, such that, 
\begin{equation}
\left( \frac{\p \urh}{\p t}, \vrh \right) + a\left(\urh, \vrh\right) + \left(N(\urh), \vrh \right) = ( f, \vrh ),\qquad \forall\, \vrh\in V_r^h, \label{eq:podfem} 
\end{equation}
and 
\begin{equation*}
\urh(x, 0) = u^h_{0, r}(x) \in V_r^h,\qquad \forall\, x\in \Om. 
\end{equation*}
Represented by $\mN_{FE}$ the nonlinear term in \eqref{eq:podfem}, whose $k$-th entry is as follows.  
\begin{eqnarray}
\left(\mN_{FE}\right)_k &=& 
\Big(N\left(u_r^h(x, t)\right),\, \varphi_k^h(x)\Big)
 \nonumber \\ 
&=&
\Big(N\big(\sum\limits_{j=1}^{r}\sum\limits_{\iota=1}^{n_{dof}}{\bf Q}_{\iota j} h_\iota(x) a_j(t)\big),\,  
\sum\limits_{s=1}^{n_{dof}}{\bf Q}_{sk} h_s(x)\Big) .
\label{eq:nonlfem}
\end{eqnarray}
Obviously, the evaluation of \eqref{eq:nonlfem} requires numerical quadratures at each time step of the numerical simulation. 
Suppose the complexity for evaluating the nonlinear function $N(u)$ with $\theta$ components is $\mathcal{O}(\varrho(\theta))$ and $n_q$ quadrature points are used in each integral evaluation, the total complexity in assembling $\mN_{FE}$ at each time step (or iteration) is around $\mathcal{O}\left(2r n_q[\varrho(2r n_{dof})+n_{dof}]\right)$ flops.

To improve the efficiency of nonlinear ROMs, we propose a method to use the FEIC for a spatial discretization of the POD-G-ROM \eqref{eq:podg}. 
The FEIC, also known as the product approximation technique \cite{christie1981product} or group finite element method \cite{fletcher1983group}, has been used as an alternative tool of the FE method for solving 
nonlinear elliptic problems \cite{lopez1988stability,sanz1984interpolation}, 
nonlinear parabolic problems \cite{chen1989error,chen1991approximation,douglas1975effect,larsson1989interpolation}, 
and nonlinear hyperbolic problems \cite{tourigny1990product,xiong2008convergence}. 
This approach replaces the nonlinear function by its interpolant in the finite dimensional space, which leads to great savings in the computational efforts while keeping the accuracy. 
To our knowledge, this is the {\em first} time that the FEIC is applied in the POD setting with a rigorous error estimate provided. 

Define the interpolation operator $\mathcal{I}^h: C(\overline{\Om})\rightarrow S_h$. 
The interpolant of $N(\cdot)$ in the FE space satisfies 
\begin{equation}
(\mathcal{I}^h N)(u(x_i, t)) = N(u(x_i, t)),
\label{feic_inter}
\end{equation}
where $x_i$ is a node in the finite element mesh, $i = 1, \ldots, n_{dof}$. 
The finite element method with interpolated coefficients of the POD-G-ROM \eqref{eq:podg} (POD-FEIC) is the following: 
to find $u_r^h(x, t) \in V_r^h$, such that, 
\begin{equation}
\left( \frac{\p u_r^h}{\p t}, v_r^h \right) + a\left(u_r^h, v_r^h\right) + \left(\mathcal{I}^h N (u_r^h), v_r^h \right) = \left( f, v_r^h \right), \qquad \forall\, v_r^h \in V_r^h, \label{eq:podgfeic}
\end{equation}  
and 
\begin{equation}
u_r^h(\cdot, 0) = u^h_{0, r}(x) \in V^h_r, \qquad \forall x\in \Om.
\label{eq:podgfeic_t0}
\end{equation} 
Different from the standard FE discretization, the nonlinear function $N(\urh)$ in \eqref{eq:podfem} is replaced by the interpolation $\mathcal{I}^h N(\urh)$ in \eqref{eq:podgfeic}.
The $k$-th row of the nonlinear term in the new numerical discretization, $\mN_{FEIC}$, reads:  
\begin{eqnarray*}
\left(\mN_{FEIC}\right)_k &=& 
\Big( (\mathcal{I}^h N)(\urh(x, t) ),\, \varphi_k^h(x)\Big) \\ 
&=&
\Big(\sum\limits_{\iota=1}^{n_{dof}}N\big(\sum\limits_{j=1}^{r}{\bf Q}_{\iota j} a_j(t)\big)h_{\iota}(x),\,  
\sum\limits_{s=1}^{n_{dof}}{\bf Q}_{sk} h_s(x)\Big) \\
&=&
\Big(\sum\limits_{\iota=1}^{n_{dof}}h_{\iota}(x),\,  
\sum\limits_{s=1}^{n_{dof}}{\bf Q}_{sk} h_s(x)\Big)\, N\Big(\sum\limits_{j=1}^{r}{\bf Q}_{\iota j} a_j(t)\Big), 
\end{eqnarray*}
which can then be rewritten as follows. 
\begin{equation}
\mN_{FEIC} =   {\bf Q}^\intercal\, {\bf M}^h\, N\big({\bf Q}\, {\bf a}(t)\big), \label{eq:nonlfeic}
\end{equation}
where ${\bf M}^h$ is the FE mass matrix with $\left[{\bf M}^h\right]_{\iota s}=(h_{\iota}, h_s)$. 

Compare with \eqref{eq:nonlfem}, the advantage of the new approach is clear: {\em the matrix ${\bf Q}^\intercal\, {\bf M}^h$ in \eqref{eq:nonlfeic} only needs to be computed once. At each time step or iteration, only values of the nonlinear function need to be calculated without involving any numerical quadratures.}
The total complexity in each time step (or iteration) is reduced to $\mathcal{O}(\varrho(2r n_{dof})+2rn_{dof})$ flops. 
The efficiency can be further improved by combining with the DEIM, which will be discussed in Section \ref{sec:FEIC-DEIM}.

%% file: Analysis.tex
\section{Error Estimates}\label{sec:analysis}
In this section, we analyze the numerical error of the POD-FEIC model \eqref{eq:podgfeic}. 
Since the obvious distinction between the POD-FEIC \eqref{eq:podgfeic} and the POD-FEM \eqref{eq:podfem} lies in the special spatial discretization of the nonlinear term, we will consider the semidiscrete model (continuous in time) and only focus on errors caused by the spatial discretization and the POD truncation.
To provide the analysis, we proceed in three steps: 
We begin by gathering a few preliminary results that will be used;  
We then prove an error estimate for the $L^2$ projection of $u$ in Lemma \ref{lemma_ritz};  
Finally, we establish the approximation error of \eqref{eq:podgfeic} in Theorem \ref{thm:err}. 

For clarity of notation, in the sequel, we will denote by $C$, a generic constant that does not depend on the mesh size $h$ and the number of POD basis functions $r$ in the ROM. 

\subsection{Step 1: Preliminary Results}
To derive the error estimation, we first make a few necessary hypotheses and present some preliminary results. 
For the solution $u$ and the nonlinear function $N(u)$ of \eqref{eq:pdewf1}, we assume: 
\begin{enumerate}[(H1)]
\item The solution $u$ belongs to $C^1(0, T; H^{m+1}(\Om)\bigcap H^1_0(\Om))$,
\item $N(u)$ belongs to $C(0, T; H^{m+1}(\Om))$, 
\item $N(u)$ is locally Lipschitz, that is, let $M= \|u\|_{L^{\infty}(\Om)}+1$, there exists $L = L(M)$ such that 
\begin{equation}
|N(u) - N(v)|\leq L\, |u-v|, \quad \text{for all } u, v \in (-M, M).
\label{eq:lipschitz}
\end{equation}
\end{enumerate}
 
Based on the FE method theory, we have the following interpolation error \cite{BS02}: 
\begin{lemma}
For $0 \leq \gamma \leq m+1$ and $1\leq p\leq \infty$, if $v\in C(\overline{\Om})\bigcap \prod\limits_{K\in \mathcal{T}_h} W^{m+1}_p(K)$, there exists a constant $C$ independent of $h$ such that
\begin{equation}
\| v - \mathcal{I}^h v \|_{\gamma, p}
\leq C \, h^{m+1-\gamma} \, \| v \|_{m+1, p}.
\label{eq:feminterror}
\end{equation}
\label{assumption_approximability}
\end{lemma}

The snapshots in our considerations are composed of  the FE solutions $u^h(x, t)$, which solves the following approximation problem: To find $u^h(x, t) \in V^h$, such that, 
\begin{equation}
\left( \frac{\p u^h}{\p t}, v^h \right) + a\left(u^h, v^h\right) + \left(N(u^h), v^h \right) = ( f, v^h ),\qquad \forall\, v^h\in V^h, \label{eq:fem1} 
\end{equation}
with $u^h(x, 0) = u^h_0(x) \in V^h$, $\forall x\in \Om$. 
The FE approximation theory for semilinear parabolic equations is well-developed (see, e.g., Chapter 14 in \cite{thomee2006galerkin} and reference therein). One can easily modify the proof of Theorem 14.1 in \cite{thomee2006galerkin} and obtain the following error estimate for the FE solution. 
\begin{lemma}
Let $u^h$ and $u$ be the solutions of \eqref{eq:fem1} and \eqref{eq:pdewf1} under the assumptions of (H1) and (H3), respectively. With appropriately chosen $u_0^h$ in the FE approximation, we have , with $C = C(u, T)$,  
\begin{equation}
\|u^h(t) - u(t)\| + h\|u^h(t)-u(t)\|_1 \leq Ch^{m+1}, \quad \text{for } t \in [0, T].
\label{eq:femerror}
\end{equation}
\end{lemma}

To estimate the interpolation of nonlinear terms in the FEIC, we utilize an auxiliary 'euclidean' norm introduced in \cite{tourigny1990product} on $C(\overline{\Om})$:
\begin{equation}
\|\chi\|_h = \left[\,\sum_{i=1}^{n_{dof}}|\chi(x_i)|^2\,\right]^{\frac{1}{2}}.
\label{def:newnorm}
\end{equation} 
For any $\chi\in V_h$, since the space is of finite dimensions, we have the equivalence between $\|\chi\|$ and $\|\chi\|_h$ on the reference element. 
With a straightforward homogeneity argument (or scaling argument \cite{BS02}), we have the following lemma: 
\begin{lemma}
\label{lem:newnorm}
There exist two strictly positive constants $c_1$ and $c_2$ independent of $h$ such that 
\begin{equation}
c_1 h^{\frac{d}{2}}\|\chi\|_h \leq \|\chi\| \leq c_2 h^{\frac{d}{2}}\|\chi\|_h, 
\label{eq:newnorm}
\end{equation}
for all $\chi\in V_h$. 
\end{lemma} 
 
Recall that the FE solutions $u^h(x, t)$ are used as snapshots, $\mathcal{H}= L^2$ is considered in the POD method and $\varphi_j(x)$ is the $j$-th POD basis. 
Besides the POD projection error in $L^2(0, T, L^2(\Om))$ given in Proposition \ref{pr:poderr}, 
\begin{equation}
\int_0^T
  \Big\| u^h(\cdot, t) - 
  \sum_{j=1}^r \left( u^h(\cdot, t), \varphi_j(\cdot) \right) \, \varphi_j(\cdot) 
  \Big\|^2 \, \id t
  = \sum\limits_{j>r} \lambda_j,
 \label{eq:poderrL2} 
\end{equation}
we have the projection error in $L^2(0, T, H^1(\Om))$ norm as follows. 
\begin{lemma}[\cite{Singler2013}]
The POD projection error in $H^1$ norm satisfies
\begin{equation}
\int_0^T
  \Big\| u^h(\cdot, t) - 
  \sum_{j=1}^r \left( u^h(\cdot, t), \varphi_j(\cdot) \right) \, \varphi_j(\cdot) 
  \Big\|_{1}^2 \, \id t
  = \sum\limits_{j>r} \|\varphi_j\|_{1}^2\, \lambda_j .
 \label{eq:poderrH1} 
\end{equation}
\label{pr:poderrH1}
\end{lemma}

For the POD approximation, we have the following POD inverse estimate \cite{KV01}:
\begin{lemma}
Let $M_{r} \in \R^{r \times r}$ with $[M_r]_{jk} = (\varphi_j , \varphi_k)$ be the POD mass matrix, 
$S_{r} \in \R^{r \times r}$ with $[S_r]_{jk} = [M_r]_{jk} + (\nabla \varphi_j , \nabla \varphi_k)$ be the POD stiffness matrix, 
and  $\| \cdot \|_2$  denote the matrix spectral norm.
Then, for all $v \in V_r$, the following estimates hold.
\begin{eqnarray}
\| v \| 
&\leq& \sqrt{\| M_r \|_2 \, \| S_r^{-1} \|_2 } \, \| v \|_{1} \, ,
\label{lemma_inverse_pod_1} \\
\| v \|_{1} 
&\leq& \sqrt{\| S_r \|_2 \, \| M_r^{-1} \|_2 } \, \| v \| \, .
\label{lemma_inverse_pod_2}
\end{eqnarray}
Since we choose $\mathcal{H}= L^2$ in the POD method, both $\| M_r \|_2$ and $\| M_r ^{-1}\|_2$ are one.
\label{lemma_inverse_pod}
\end{lemma}
\subsection{Step 2: $L^2$ Projection Error}
Next, we define the $L^2$ projection of $u$, $w_r^h$, from $L^2$ to $V_r^h$ such that  
\begin{equation}
\left( u- w_r^h, \vrh\right) = 0,\qquad  \forall\, \vrh\in V_r^h. 
\label{eq:L2proj}
\end{equation}
We have the following estimation of the $L^2$ projection error. 
\begin{lemma}
The $L^2$ projection of $u$, $w_r^h$, satisfies the following error estimations:
\begin{eqnarray}
\int_0^T \left\| u-\wrh \right\|^2 \, \id t
&\leq& C\left(h^{2m+2}  + \sum_{j > r} \lambda_j\right), 
\label{eq:L2} \\
\int_0^T \left \| \nabla \left( u - \wrh \right) \right \|^2\, \id t
&\leq& 
C\left(h^{2m}+ \|S_r\|_2\, h^{2m+2}+ \sum_{j >r}\|\varphi_j\|_{1}^2\, \lambda_j \right)\,,
\label{eq:L2_grad}
\end{eqnarray}
where $C = C(u, T)$.
\label{lemma_ritz}
\end{lemma}   

\begin{proof}
\begin{eqnarray}
\left\| u-\wrh \right\|^2
&=& \left( u - \wrh, u-\wrh \right) \nonumber \\
&\stackrel{ \eqref{eq:L2proj} }{=}& \left( u - \wrh, u-\vrh \right), \qquad \forall \, \vrh \in V_r^h. 
\label{lemma_L2_1}
\end{eqnarray}

It indicates, by Cauchy-Schwartz inequality, for all $\vrh \in V_r^h$,  
\begin{equation}
\left\| u-\wrh \right\| \leq \|u-\vrh\|.
\label{lemma_L2_2}
\end{equation}
Decomposing $u-\vrh = u- u^h + (u^h-\vrh)$ and choosing $\vrh = \textrm{P}_r u^h = \sum\limits_{j=1}^r \left( u^h, \varphi_j \right)\varphi_j $ in \eqref{lemma_L2_2}, by the triangular inequality and Proposition \ref{pr:poderr}, we have
\begin{eqnarray}
\int_0^T \left\| u-\wrh \right\|^2\, \id t
&\leq& C\left( \int_0^T \left\| u- u^h \right\|^2\, \id t 
+ \int_0^T \Big\| u^h - \sum\limits_{j=1}^r \left( u^h, \varphi_j \right)\varphi_j \Big\|^2 \, \id t \right) \nonumber \\
&\stackrel{ \eqref{eq:poderrL2} }{\leq}& 
C\left(\int_0^T \left\| u-u^h \right\|^2 \, \id t  + \sum_{j > r} \lambda_j\right)\,.
\label{lemma_L2_3}
\end{eqnarray}
Considering the FE approximation error \eqref{eq:femerror}, we have the bound for the $L^2$ projection error in $L^2(0, T; L^2(\Om))$ as \eqref{eq:L2}. 

By the triangular inequality and adding and subtracting $\textrm{P}_r u^h= \sum\limits_{j=1}^r \left( u^h, \varphi_j \right)\varphi_j $, we have
\begin{eqnarray*}
\int_0^T \left \| \nabla \left( u - \wrh \right) \right \|^2\, \id t
&\leq C&
\int_0^T \left(
\left\| \nabla\left(u- u^h\right) \right\|^2 + \left\| \nabla\left(u^h-\textrm{P}_r u^h\right) \right\|^2 + \left\| \nabla\left(\textrm{P}_r u^h-\wrh\right) \right\|^2  \right)  \id t  \\
&\stackrel{\eqref{eq:poderrH1}, \eqref{lemma_inverse_pod_2}}{\leq} C&
\left( \int_0^T \left\| \nabla\left(u- u^h\right) \right\|^2 \id t + \sum\limits_{j>r} \|\varphi_j\|_{1}^2\,\lambda_j  
+ \left\| S_r \right\|_2\int_0^T \|\textrm{P}_r u^h - \wrh \|^2 \right) \\
&\leq C&
\left( \int_0^T \left\| \nabla\left(u- u^h\right) \right\|^2 \id t + \sum\limits_{j>r} \|\varphi_j\|_{1}^2\,\lambda_j  
+ \left\| S_r \right\|_2\int_0^T \|u-u^h\|^2 \right),
\label{eq:errlem6} 
\end{eqnarray*}
where we use $\|\wrh- \textrm{P}_r u^h\| \leq \|u-u^h\|$ in the last inequality. 
Considering the FE approximation error estimation \eqref{eq:femerror}, we have the bound for the $L^2$ projection error in $L^2(0, T; H^1(\Om))$ as \eqref{eq:L2_grad}. 
This proved the lemma. 
\end{proof}

\subsection{Step 3: Main Results}
Finally, we discuss the main theoretical results of this paper, which is about the approximation property of the POD-FEIC \eqref{eq:podgfeic}. 
We first estimate the difference between the $L^2$ projection of $u$, $\wrh$, and the approximation solution $\urh$ on a certain time interval in Lemma \ref{lem:prh_error}, which is bounded by the $L^2$ projection error. 
The conclusion is then extended to the whole time interval through a continuity argument in Lemma \ref{lem:prh_error_T}.  
By using the triangular inequality, we get the final estimation of approximation error $u-\urh$ in Theorem \ref{thm:err}. 
The crucial component  in the analysis is the interpolation error of the nonlinear term $\left( N(u) - \mathcal{I}^h N (\urh), \vrh \right)$, which can be decomposed in two different ways: 
\begin{equation}
\left(N(u) - \mathcal{I}^h N (\urh), \vrh \right) = 
\left(N(u)- \mathcal{I}^h N (u), \vrh \right) + \left(\mathcal{I}^h N (u) - \mathcal{I}^h N (\urh), \vrh \right),
\label{eq:nonlapproach1}
\end{equation}
and 
\begin{equation}
\left(N(u) - \mathcal{I}^h N (\urh), \vrh \right) = 
\left(N(u)- N (\urh), \vrh \right) + \left(N (\urh) - \mathcal{I}^h N (\urh), \vrh \right). 
\label{eq:nonlapproach2}
\end{equation}
The first approach has been used in \cite{tourigny1990product}, 
in which the first term on the RHS of \eqref{eq:nonlapproach1} is bounded by the standard FE interpolation error under the smoothness assumption of $N(u)$, and the second term can be estimated with the help of the local Lipschitz continuity assumption of $N(u)$ and an auxiliary 'euclidean' norm. 
The second approach has been used in \cite{chen1989error,larsson1989interpolation}, in which the first term on the RHS of \eqref{eq:nonlapproach2} also appears in the standard FE discretization, thus can be treated as usual. 
The second term can be estimated by the FE interpolation error of $N(\urh)$, whose accuracy relies on the regularity of $N(\urh)$.  It is determined by a hypothesis on $\urh$ and a stronger smoothness assumption of $N(u)$ than that required in the first approach.   
In this paper, we will follow the first approach \eqref{eq:nonlapproach1}.

let $v = \vrh$ in \eqref{eq:pdewf1} and subtract it from \eqref{eq:podgfeic}, we have the error equation of $e= \urh - u$ as follows. 
\begin{equation}
\left( e_t, \vrh \right) 
+ a\left(e, \vrh\right)
+ \left(\mathcal{I}^h N (\urh)-N(u), \vrh \right) = 0, \qquad \forall\, \vrh \in V_r^h.
\label{eq:err1}
\end{equation}
Let  $\prh = \urh - \wrh$ and $\eta = u - \wrh$, we have the decomposition of error, $e = \prh - \eta$. 
Based on the definition of $L^2$ projection \eqref{eq:L2proj}, we have $(\eta_t, \vrh) = 0$. 
Therefore, the error equation \eqref{eq:err1} can be rewritten as: 
\begin{equation}
\left( \prth, \vrh \right) + a\left(\prh, \vrh\right) =  a\left(\eta, \vrh\right) + \left( N(u) - \mathcal{I}^h N(\urh), \vrh \right) , \qquad \forall\, \vrh \in V_r^h.
\label{eq:err2}
\end{equation}
\begin{lemma}
Under assumptions (H1)-(H3), let $u$ be the solution of \eqref{eq:pdewf1} and $\urh$ be the solution of \eqref{eq:podgfeic} with the initial condition $\urh(\cdot, 0)= \sum_{j=1}^r (u_0, \varphi^h_{r, j})\varphi^h_{r, j}$ in \eqref{eq:podgfeic_t0}, respectively. 
Assume that, for some $t_1$ with $0< t_1 \leq T$, we have 
\begin{equation}
\|\prh(t)\|_{L^{\infty}(\Om)}<\frac{1}{2},\quad \text{for all } 0<t\leq t_1.
\label{eq:assume_prh}
\end{equation}
Then it follows that 
\begin{equation}
\|\prh(t_1)\|^2+\int_{0}^{t_1}\|\nabla \prh(t)\|^2  \id t \leq C(u, T) \left(h^{2m} + \|S_r\|_2\, h^{2m+2} + \sum\limits_{j>r} \|\varphi_j\|_{1}^2\,  \lambda_j\right).
\label{eq:prh_error_t1}
\end{equation}
\label{lem:prh_error}
\end{lemma}

\begin{proof}
First note that one can certainly find a $t_1$ to make \eqref{eq:assume_prh} true by adjusting the mesh size $h$ and the number of POD basis functions $r$. 
We choose $h$ to be small enough and $r$ to be large enough to ensure that $\|u-\wrh\|_{L^{\infty}(\Om)}<\frac{1}{2}$, then together with \eqref{eq:assume_prh}, we have both $\wrh$ and $\urh$ locate on the interval $(-M, M)$, where $M$ is defined in (H3). 
This allows us to take advantage of the local Lipschitz condition of $N(u)$ on the interval $t\in [0, t_1]$. Next we start estimations. 

Let $\vrh = \prh$ in \eqref{eq:err2}, we have 
\begin{eqnarray}
\left( \prth, \prh \right) &+& a\left(\prh, \prh\right) =  a\left(\eta, \prh\right) + \left( N(u) - \mathcal{I}^h N(u), \prh \right)   \nonumber \\
&+& \left(  \mathcal{I}^h N(u) - \mathcal{I}^h N(\wrh), \prh \right)  
+ \left(  \mathcal{I}^h N(\wrh) - \mathcal{I}^h N(\urh), \prh \right). 
\label{eq:err3}
\end{eqnarray}
By the continuity \eqref{a_continuity} and coercivity \eqref{a_coercivity} of $a(\cdot, \cdot)$, and Cauchy-Schwartz inequality, we have 
\begin{eqnarray}
\frac{1}{2} \frac{\id}{\id t}\|\prh\|^2 
&+& \beta \|\nabla \prh\|^2 
\leq
\alpha \|\nabla \eta\|\, \|\nabla \prh\|
+ \|N(u) - \mathcal{I}^h N(u)\|\, \|\prh\| \nonumber \\
&+& \|\mathcal{I}^h N(u) - \mathcal{I}^h N(\wrh)\|\, \| \prh\|
+ \|\mathcal{I}^h N(\wrh) - \mathcal{I}^h N(\urh)\|\, \|\prh\|.
\label{eq:err4}
\end{eqnarray}
For the first term on the RHS of \eqref{eq:err4}, by the Youngs' inequality, we have 
\begin{eqnarray}
\alpha \|\nabla \eta\|\, \|\nabla \prh\| 
\leq 
\frac{\alpha^2}{2 \beta}\|\nabla \eta\|^2 + \frac{\beta}{2}\|\nabla \prh\|^2. 
\label{eq:err41}
\end{eqnarray}
For the second term on the RHS of \eqref{eq:err4}, Young's inequality yields 
\begin{eqnarray}
\|N(u) - \mathcal{I}^h N(u)\|\, \| \prh\| \leq
\frac{\|N(u) - \mathcal{I}^h N(u)\|^2}{2} + \frac{\|\prh\|^2}{2}.
\label{eq:err420}
\end{eqnarray}
Considering the FE interpolation error and the assumption (H2), we have 
\begin{eqnarray}
\|N(u) - \mathcal{I}^h N(u)\| \leq C h^{m+1} |N(u)|_{H^{m+1}(\Om)}.
\label{eq:err42}
\end{eqnarray}
For the third term on the RHS of \eqref{eq:err4}, Young's inequality gives 
\begin{eqnarray}
\|\mathcal{I}^h N(u) - \mathcal{I}^h N(\wrh)\|\, \| \prh\|
\leq
\frac{\|\mathcal{I}^h N(u) - \mathcal{I}^h N(\wrh)\|^2}{2} + \frac{\|\prh\|^2}{2}. 
\label{eq:err430}
\end{eqnarray}
Employing the Lemma \ref{lem:newnorm}, the local Lipschitz condition, and the triangular inequality, we have 
\begin{eqnarray}
\|\mathcal{I}^h N(u) - \mathcal{I}^h N(\wrh)\|  
&\stackrel{\eqref{eq:newnorm}}{\leq}& c_2 h^{\frac{d}{2}}\|\mathcal{I}^h N(u) - \mathcal{I}^h N(\wrh)\|_h
\stackrel{\eqref{feic_inter},\eqref{def:newnorm}}{=} c_2 h^{\frac{d}{2}}\|N(u) - N(\wrh)\|_h
\nonumber \\
&\stackrel{\eqref{eq:lipschitz}}{\leq}& L c_2 h^{\frac{d}{2}}\|u - \wrh \|_h
\stackrel{\eqref{feic_inter}, \eqref{def:newnorm}}{=} L c_2 h^{\frac{d}{2}}\|\mathcal{I}^h u - \wrh \|_h \nonumber \\
&\stackrel{\eqref{eq:newnorm}}{\leq}& L c_2 c_1^{-1}\|\mathcal{I}^h u - \wrh \| \nonumber \\
&\leq& L c_2 c_1^{-1}\left( \|u - \mathcal{I}^h u\| + \| u - \wrh \| \right). 
\label{eq:err43_0}
\end{eqnarray}
By the regularity assumption (H1) of solution $u$ and the FE approximability (Lemma \ref{assumption_approximability}), we have
\begin{eqnarray}
 \|u - \mathcal{I}^h u\| \leq C h^{m+1} |u|_{H^{m+1}(\Om)}.
 \label{eq:err43_1}
\end{eqnarray}
Thus, we have, for the third term on the RHS of \eqref{eq:err4}, 
\begin{eqnarray}
\|\mathcal{I}^h N(u) - \mathcal{I}^h N(\wrh)\|  
\leq
C\left(h^{m+1} + \|\eta\|\right).
\label{eq:err43}
\end{eqnarray}
For the forth term on the RHS of \eqref{eq:err4}, we use similar arguments to those for the third term and get 
\begin{eqnarray}
\|\mathcal{I}^h N(\wrh) - \mathcal{I}^h N(\urh)\|  
&\stackrel{\eqref{eq:newnorm}}{\leq}& c_2 h^{\frac{d}{2}}\|\mathcal{I}^h N(\wrh) - \mathcal{I}^h N(\urh)\|_h 
\stackrel{\eqref{feic_inter}, \eqref{def:newnorm}}{=}  c_2 h^{\frac{d}{2}}\|N(\wrh) - N(\urh)\|_h \nonumber \\
&\stackrel{\eqref{eq:lipschitz}}{\leq}& L c_2 h^{\frac{d}{2}}\|\wrh - \urh\|_h
\stackrel{\eqref{eq:newnorm}}{\leq} L c_2 c_1^{-1}\|\wrh - \urh\| \nonumber \\
&=& L c_2 c_1^{-1}\|\prh\|.
\label{eq:err44}
\end{eqnarray}
Substituting \eqref{eq:err41}, \eqref{eq:err420}, \eqref{eq:err42},  \eqref{eq:err430}, \eqref{eq:err43}, and \eqref{eq:err44} in \eqref{eq:err4}, we obtain 
\begin{eqnarray}
\frac{\id}{\id t}\|\prh\|^2 
+ \beta \|\nabla \prh\|^2 
&\leq& 
\frac{\alpha^2}{\beta} \|\nabla \eta\|^2
+ C\left(h^{2m+2}+\|\eta\|^2\right) + C^* \|\prh\|^2, 
\label{eq:err5}
\end{eqnarray}
where $C_*=L c_2 c_1^{-1}+1$. 
By Gronwall's lemma, on the interval $[0, t_1]$, we have 
\begin{eqnarray}
\|\prh(t_1)\|^2 
&+& \int_0^{t_1} \beta \|\nabla \prh\|^2 \id s  
\leq
\|\prh(0)\|^2  \nonumber \\
&&\hspace{1cm} + e^{C_* t_1}\int_0^{t_1} \left(\frac{\alpha^2}{\beta} \|\nabla \eta\|^2
+ Ch^{2m+2}+ C\|\eta\|^2 \right) \id s.
\label{eq:err51}
\end{eqnarray} 
Considering $t_1\leq T$ and the choice of the initial condition which indicates $\|\prh(0)\|= 0$, the above inequality yields 
\begin{eqnarray}
\|\prh(t_1)\|^2 
+ \int_0^{t_1}\beta \|\nabla \prh\|^2 \id s  
&\leq&
e^{C_* T} \int_0^{T} \left(\frac{\alpha^2}{\beta} \|\nabla \eta\|^2
+ C h^{2m+2} + C\|\eta\|^2 \right) \id s, \nonumber \\
&\stackrel{\eqref{eq:L2_grad}}{\leq}&
C\left(h^{2m}+ \|S_r\|_2\, h^{2m+2}+ \sum_{j >r} \|\varphi_j\|_{1}^2\,  \lambda_j \right)\,.
\label{eq:err6}
\end{eqnarray} 
where $C= C(u, T)$ independent of $t_1$. 
\end{proof}
%
\begin{lemma}
Suppose the order of finite elements $m\geq 1$ for $d=1$ and $m\geq 2$ for $d\geq 2$, respectively. With the same conditions as those in Lemma \ref{lem:prh_error}, we have 
\begin{equation}
\|\prh(T)\|^2+\int_{0}^{T}\|\nabla \prh(t)\|^2  \id t \leq C(u, T) \left(h^{2m} + \|S_r\|_2\, h^{2m+2} + \sum_{j >r} \|\varphi_j\|_{1}^2\,  \lambda_j \right)\, .
\label{eq:prh_error_T}
\end{equation}
\label{lem:prh_error_T}
\end{lemma}
\begin{proof}
Assume $t_1^*$ is the largest value that makes \eqref{eq:assume_prh} true. 
If $t_1^* \neq T$, it must be 
\begin{equation}
\|\prh(t_1^*)\|_{L^{\infty}(\Om)} = \frac{1}{2}. 
\label{eq:prh_infty}
\end{equation}
However, by the inverse inequality, 
\begin{equation}
\|\prh(t_1^*)\|_{L^{\infty}(\Om)}\leq C h^{-\frac{d}{2}}\|\prh(t_1^*)\| 
\leq
C h^{-\frac{d}{2}} \left(h^m +\sqrt{\|S_r\|_2}\,  h^{m+1} +\sqrt{\sum_{j >r} \|\varphi_j\|_{1}^2\,  \lambda_j}\right). 
\end{equation}
Then, for $m\geq 1$ if $d=1$, and for $m\geq 2$ if $d= 2, 3$, one can always find a $h$ small enough and a $r$ large enough such that $\|\prh(t_1^*)\|_{L^{\infty}(\Om)} < \frac{1}{2}$. 
This contradicts with the assumption \eqref{eq:prh_infty}. Therefore, $t_1^* = T$, that is, the conclusion \eqref{eq:prh_error_t1} is true on the whole time interval $[0, T]$.
\end{proof}
\begin{theorem}
Under assumptions (H1)-(H3), let $u$ be the solution of \eqref{eq:pdewf1} and $\urh$ be the solution of \eqref{eq:podgfeic} with the initial condition $\urh(\cdot, 0)= \sum_{j=1}^r (u_0, \varphi^h_{r, j})\varphi^h_{r, j}$ in \eqref{eq:podgfeic_t0}, respectively.  
The order of finite elements $m\geq 1$ for $d=1$, $m\geq 2$ for $d\geq 2$. 
There exist positive numbers $h_0$ and $r_0$ such that, for $h\leq h_0$ and $r\geq r_0$, we have 
\begin{equation}
\|u^h_{r}(T)-u(T)\|^2 + \int_{0}^T \|\nabla \urh(t) - \nabla u(t) \|^2 \id t \leq C\left(h^{2m} + \|S_r\|_2\,  h^{2m+2} + \sum_{j >r} \|\varphi_j\|_{1}^2\,  \lambda_j\right), 
\label{eq:err}
\end{equation}
where $C= C(u, T) $ is a positive constant independent of $h$ and $\lambda_j$. 
\label{thm:err}
\end{theorem}
\begin{proof}
The conclusion follows the triangular inequality, Lemma \ref{lem:prh_error_T}, and Lemma \ref{lemma_ritz}. 
\end{proof}
\begin{remark}
In Theorem \ref{thm:err}, the estimation for the gradient of errors, $\|\nabla \urh - \nabla u\|_{L^2(0, T; L^2(\Om))}$ is optimal with respect to (w.r.t.) both the spatial discretization ($\sim \mathcal{O}(h^m)$) and the POD truncation ($\sim \mathcal{O}(\sqrt{\sum_{j >r} \|\varphi_j\|_{1}^2\,  \lambda_j})$). 
By H\"older's inequality, it is easy to show the error in $L^2(0, T; H^1(\Om))$ has the same optimality. 
Although not proven here, the error in $L^2(0, T; L^2(\Om))$ is also optimal w.r.t. the spatial discretization ($\sim \mathcal{O}(h^{m+1})$) and the POD truncation ($\sim \mathcal{O}(\sqrt{\sum_{j >r} \lambda_j})$).  
We shall verify it numerically in the next section.  
\end{remark}
\begin{remark}
\label{rem:pdewf2}
Theorem \ref{thm:err} also holds on the FEIC approximation of the POD-G-ROM for the problem \eqref{eq:pdewf2}. The proof follows a similar argument. 
\end{remark}

%% file: Numerical.tex
\section{Numerical Tests}\label{sec:numeric}
The goal of this section is twofold: 
first, the proposed method will be validated by two problems appeared in interdisciplinary research. 
Accuracy and efficiency of the new approach are to be tested; 
second, the theoretical result in Section \ref{sec:analysis} will be numerically verified by another example. 

We mainly compare the results of the POD-FEIC \eqref{eq:podgfeic} with those of the POD-FEM \eqref{eq:podfem}. 
The errors in two different norms are measured: 
$$
\mathcal{E}_0(u, v) = \sqrt{\frac{1}{M}\sum\limits_{\ell=1}^M \|u(\cdot, t_{\ell}) - v(\cdot, t_{\ell}) \|^2}\, ,\quad
\mathcal{E}_1(u, v) = \sqrt{\frac{1}{M}\sum\limits_{\ell=1}^M \|u(\cdot, t_{\ell}) - v(\cdot, t_{\ell}) \|_1^2}\, ,
$$
where $\mathcal{E}_0$ is a discrete approximation of the error in $L^2(0, T; L^2(\Om))$, while $\mathcal{E}_1$ is a discrete approximation of the error in $L^2(0, T; H^1(\Om))$. 
For a fair comparison, all numerical tests reported in this paper are implemented on a PC with a 2.6GHz Intel Core i7 processor. 
As a criterion for efficiency, the CPU time of simulations is considered, which is the (on-line) time elapsed for the integration only, excluding the (off-line) time for generating basis functions, precomputing matrices, and calculating errors, etc.  

\subsection{Validation}\label{sec:validation}
For the validation purpose, we consider two examples: a one-dimensional (1D) FitzHugh-Nagumo (F-N) system, which possesses a cubic polynomial nonlinearity; and a two-dimensional (2D) Buckley-Leverett equation (BLE), which has a non-polynomial nonlinearity. 
\paragraph{FitzHugh-Nagumo System}
We first consider the simplified Hodgkin-Huxley model used in \cite{chaturantabut2010nonlinear}, which is a  F-N system. The model is a nonlinear PDE system and describes the activation and deactivation dynamics of a spiking neuron. The system reads: 
\begin{eqnarray}
\frac{\p v}{\p t} - \mu \Delta v -\frac{1}{\mu} v(v-0.1)(1-v) + \frac{1}{\mu} w = \frac{c}{\mu} , &\quad x\in [0, L], t\in [0, T], \\
\frac{\p w}{\p t} -b v +\gamma w = c, &\quad x\in [0, L], t\in [0, T], \nonumber \\
v(x, 0) = 0,\quad  w(x, 0) = 0, &\quad x\in [0, L], \nonumber \\
v_x(0, t) = -i_0(t),\quad v_x(L, t) = 0, &\quad t\in [0, T], \nonumber
\end{eqnarray}
where $v(x, t)$ and $w(x,t )$ are voltage and recovery voltage, respectively. 
Let $\bu = [v, w]^\intercal$, the weak form of the original system has the form of \eqref{eq:pdewf1}: 
\begin{equation}
\left( \frac{\p \bu}{\p t}, \bv \right) + {\bf M_1}\, a\left(\bu, \bv\right) + {\bf M_2}\, \left( \bu, \bv \right) + ({\bf N}, \bv) = ( \bf{f}, v ),
\end{equation}
where $a(\bu, \bv) = \mu(\nabla \bu, \nabla \bv)$, 
${\bf N}= \left[\frac{1}{\mu} v(v-0.1)(1-v), 0\right]^\intercal$, 
${\bf f}= \left[\frac{c}{\mu}, c \right]^\intercal$,
${\bf M_1} = \left[\begin{array}{cc} 
\mu& 0 \\
0    & 0
\end{array}\right]$, 
and 
${\bf M_2} = \left[\begin{array}{cc} 
0    & \frac{1}{\mu} \\
-b    & \gamma
\end{array}\right]$.
We choose the same parameters as those utilized in \cite{chaturantabut2010nonlinear}, that is, 
$L= 1$, $T= 8$, $\mu= 0.015$, $b= 0.5$, $\gamma= 2$, $c= 0.05$ and the stimulus $i_0(t) = 50000t^3 e^{-15t}$. 

To generate snapshots, we use linear finite elements on a uniform mesh for spatial discretization with mesh size $h=1/512$, and the Crank-Nicolson scheme for time integration with time step $\Delta t = 1\times 10^{-2}$. 
Totally, 801 snapshots are collected and used to compute POD basis in $L^2$ space. 
Since the exact solution is unknown, we regard the FE solution as the benchmark. 

The errors and simulation time of the POD-FEM and the POD-FEIC are listed in Table \ref{table:fn_r} when the same number of POD basis functions, $r$, is used in both ROMs. 
It is seen that the FEIC discretization achieves same accuracy as that of the FE discretization, however, reduces the CPU time by a factor of $150$. 
A comparison of the limit cycle projected onto the $v-w$ plane among the FE solution, the POD-FEM ($r = 5$) solution, and the POD-FEIC ($r=5$) solution is shown in Figure \ref{Fig:fn_r}, which illustrates the correct limit cycle of the original system has been captured by the POD-FEIC when only 5 POD basis functions are used. 

\begin{table}[htb]
\centering
\caption{
Errors and the CPU time of the POD-FEM \eqref{eq:podfem} and the POD-FEIC \eqref{eq:podgfeic}. 
Note that the POD-FEIC keeps the same accuracy as the POD-FEM, but saves CPU time by over 150 times. 
}\label{table:fn_r}
\centering
\begin{tabular}{c|ccc|ccc}
\hline
\multirow{2}{*}{$r$}& \multicolumn{3}{c}{POD-FEM}&\multicolumn{3}{|c}{POD-FEIC}\\
\cline{2-7}
{}&$\mathcal{E}_0(\urh, u^h) $ &$\mathcal{E}_1(\urh, u^h) $ & CPU time & $\mathcal{E}_0(\urh, u^h) $ & $\mathcal{E}_1(\urh, u^h) $ & CPU time \\
\hline
 3 & 3.64e-03& 9.96e-02& 2.09e+02 & 3.64e-03& 9.96e-02 &1.20 \\
 5 & 6.61e-04& 2.27e-02& 2.11e+02 & 6.60e-04& 2.27e-02 &1.23  \\
 7 & 1.20e-04& 4.91e-03& 2.12e+02 & 1.20e-04& 4.91e-03 &1.25  \\
 9 & 2.03e-05& 1.25e-03& 2.12e+02 & 2.10e-05& 1.25e-03 &1.27  \\
\hline
\end{tabular}
\end{table} 
\begin{figure}[htb]
\centering
\includegraphics[width=0.5\textwidth]{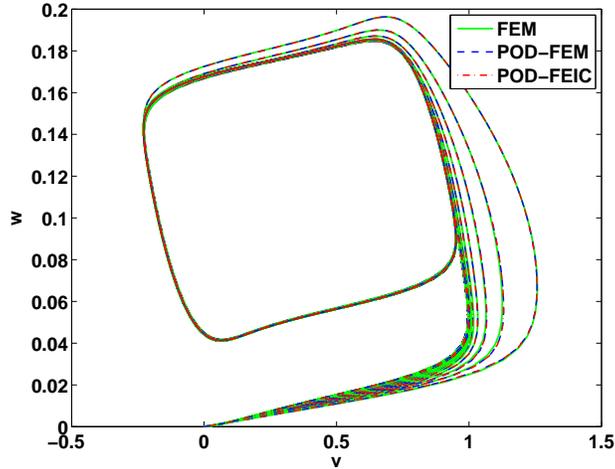}
\caption{
The limit cycle on the $v-w$ plane: the POD approximation ($r=5$) with either the FEM discretization (POD-FEM) or the FEIC discretization (POD-FEIC) coincides with that of the finite element solution (FEM).  
}\label{Fig:fn_r}
\end{figure} 

\paragraph{Buckley-Leverett Equation}  
We then consider the 2D BLE, which is usually used to describe two phase flow in porous media with a gravitation pull in $x$-direction \cite{nguyen2008efficient}. 
\begin{eqnarray}
\frac{\p u}{\p t} - \mu \Delta u + \frac{\p f_1(u)}{\p x} + \frac{\p f_2(u)}{\p y} = 0  , &\quad x\in \Om, t\in [0, T], \label{eq:ble} \\
u(x, 0) = e^{-16(x^2+y^2)}, &\quad x\in [0, L], \nonumber \\
u(x, t) =  0, &\quad  x\in \p\Om, t\in [0, T], \nonumber
\end{eqnarray}
where $f_1(u) = \frac{u^2}{u^2+(1-u)^2}$ and $f_2(u) = f_1(u) [1-5(1-u)^2]$. 
The weak form of the original system has the form of \eqref{eq:pdewf2} with $a(u, v) = \mu(\nabla u, \nabla v)$ and $N(u) = [-f_1, -f_2]^\intercal$.

In this test, we choose the same parameters as used in \cite{nguyen2008efficient}: $\mu = 0.1$, $T= 0.5$, $\Om = [-1.5, 1.5]\times [-1.5, 1.5]$.  
To generate snapshots, we use linear finite elements on a uniform triangular mesh for spatial discretization with mesh size $h=1/64$, and the Crank-Nicolson scheme for time integration with time step $\Delta t = 1\times 10^{-2}$. 
51 snapshots are collected and used to compute POD basis in $L^2$ space. 
Due to the lack of exact solution, we also consider the FE solution to be the benchmark.

The POD-FEM and the POD-FEIC approximation erros are listed in Table \ref{table:bl_r} when $r$ POD basis functions used in both ROMs.  
It is seen that the POD-FEIC obtains the same accuracy as that of the POD-FEM, but, decreases the CPU time by 40 times. 
A comparison among the FE solution, the POD-FEM, and the POD-FEIC result at $t=0.2$ is shown in Figure \ref{Fig:bl_r}. 

\begin{table}[htb]
\centering
\caption{
Errors and the CPU time of the POD-FEM \eqref{eq:podfem} and the POD-FEIC \eqref{eq:podgfeic}. 
Note that the POD-FEIC keeps the same accuracy as the POD-FEM, but saves CPU time by over 40 times. 
}\label{table:bl_r}
\centering
\begin{tabular}{c|ccc|ccc}
\hline
\multirow{2}{*}{$r$}& \multicolumn{3}{c}{POD-FEM}&\multicolumn{3}{|c}{POD-FEIC}\\
\cline{2-7}
{}&$\mathcal{E}_0(\urh, u^h)$ &$\mathcal{E}_1(\urh, u^h)$ & CPU time & $\mathcal{E}_0(\urh, u^h)$ & $\mathcal{E}_1(\urh, u^h)$ & CPU time \\
\hline
 5  & 4.44e-03& 7.63e-02& 4.86e+03 & 4.43e-03& 7.62e-02 & 116 \\
10 & 3.58e-04& 1.03e-02& 4.92e+03 & 3.70e-04& 1.04e-02 & 118  \\
15 & 3.34e-05& 1.34e-03& 4.95e+03 & 1.04e-04& 2.29e-03 & 119  \\
\hline
\end{tabular}
\end{table} 

\begin{figure}[htb]
\centering
\begin{minipage}[ht]{0.31\linewidth}
\includegraphics[width=1\textwidth]{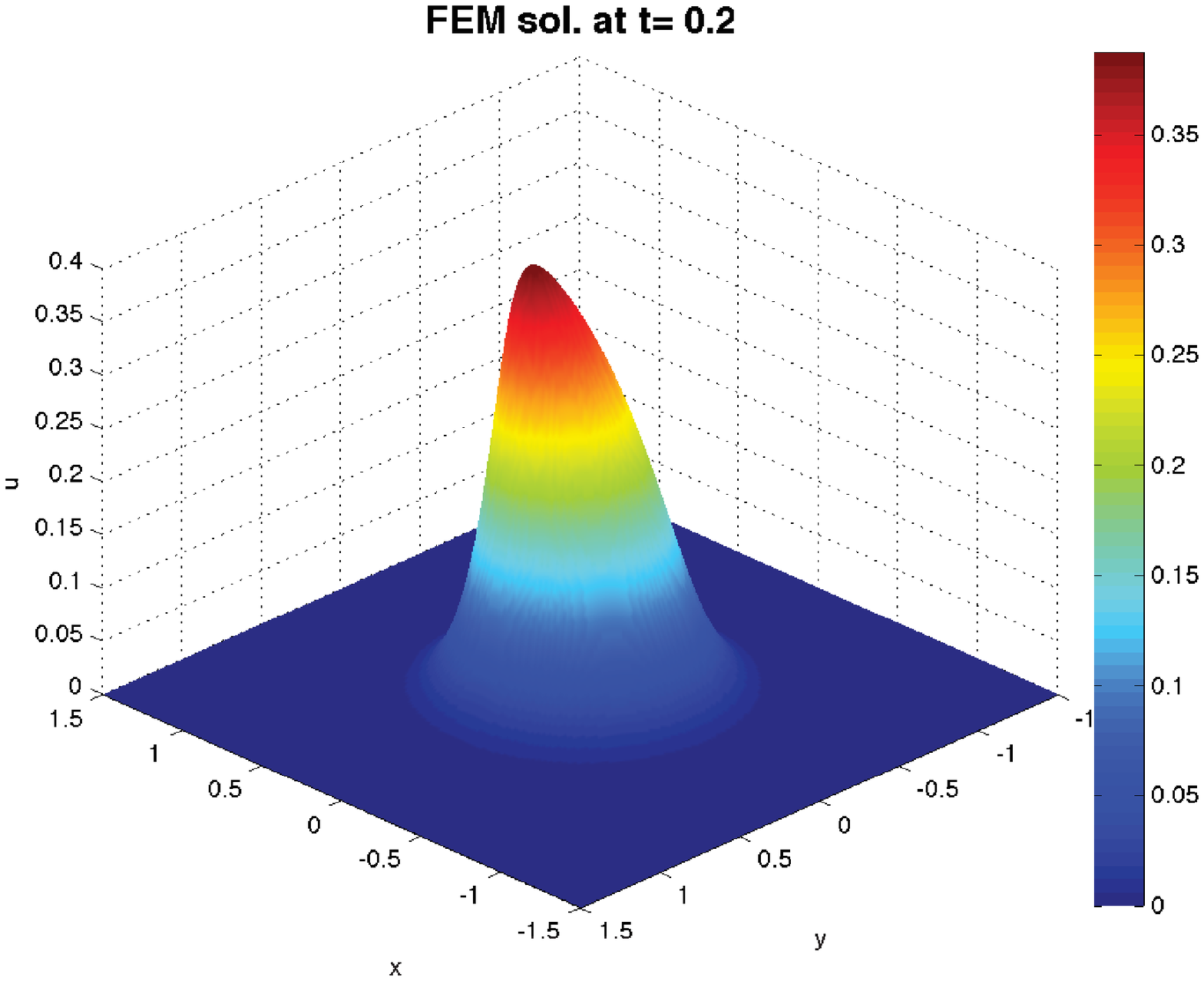}
\end{minipage}
\begin{minipage}[ht]{0.31\linewidth}
\includegraphics[width=1\textwidth]{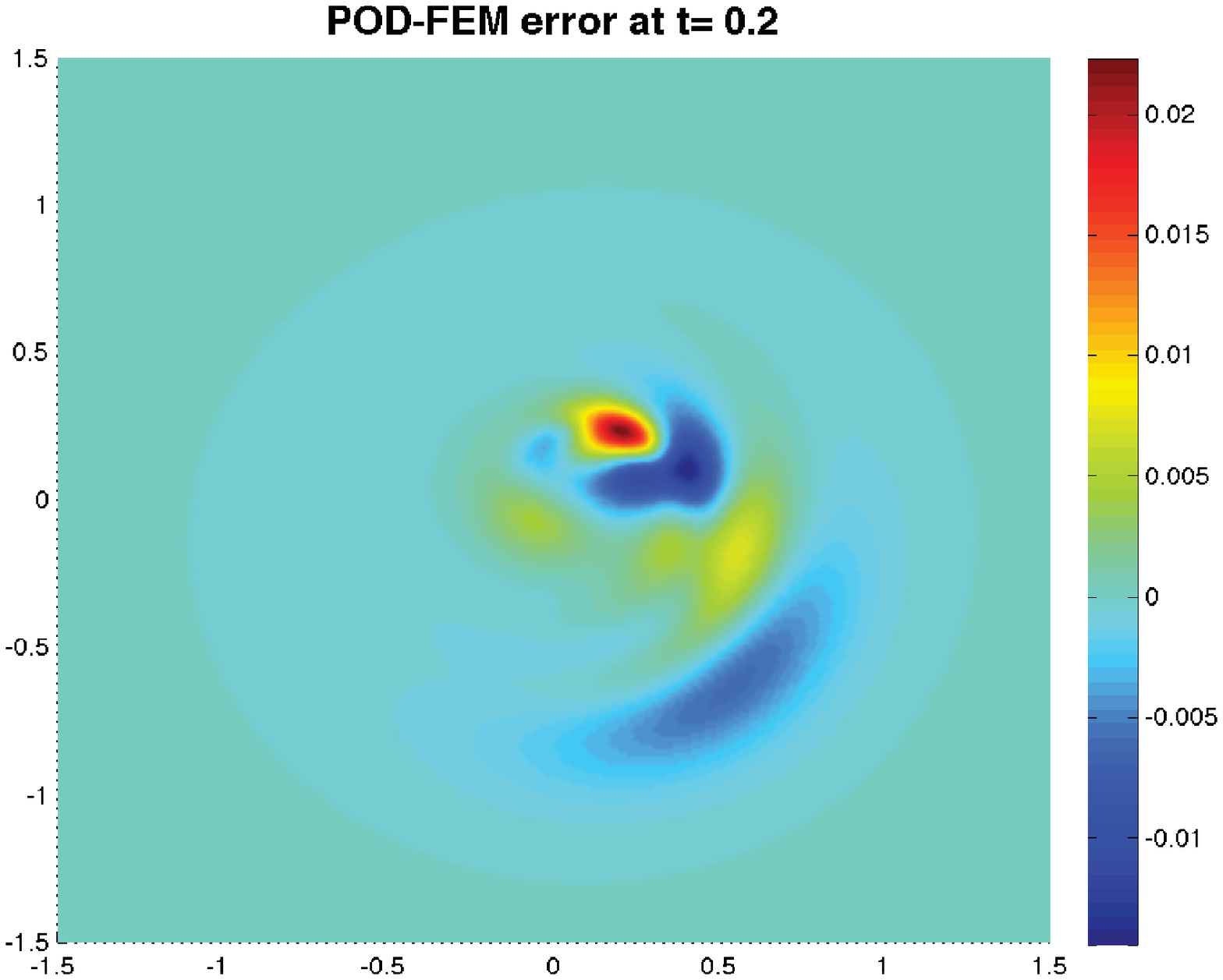}
\end{minipage}
\begin{minipage}[ht]{0.31\linewidth}
\includegraphics[width=1\textwidth]{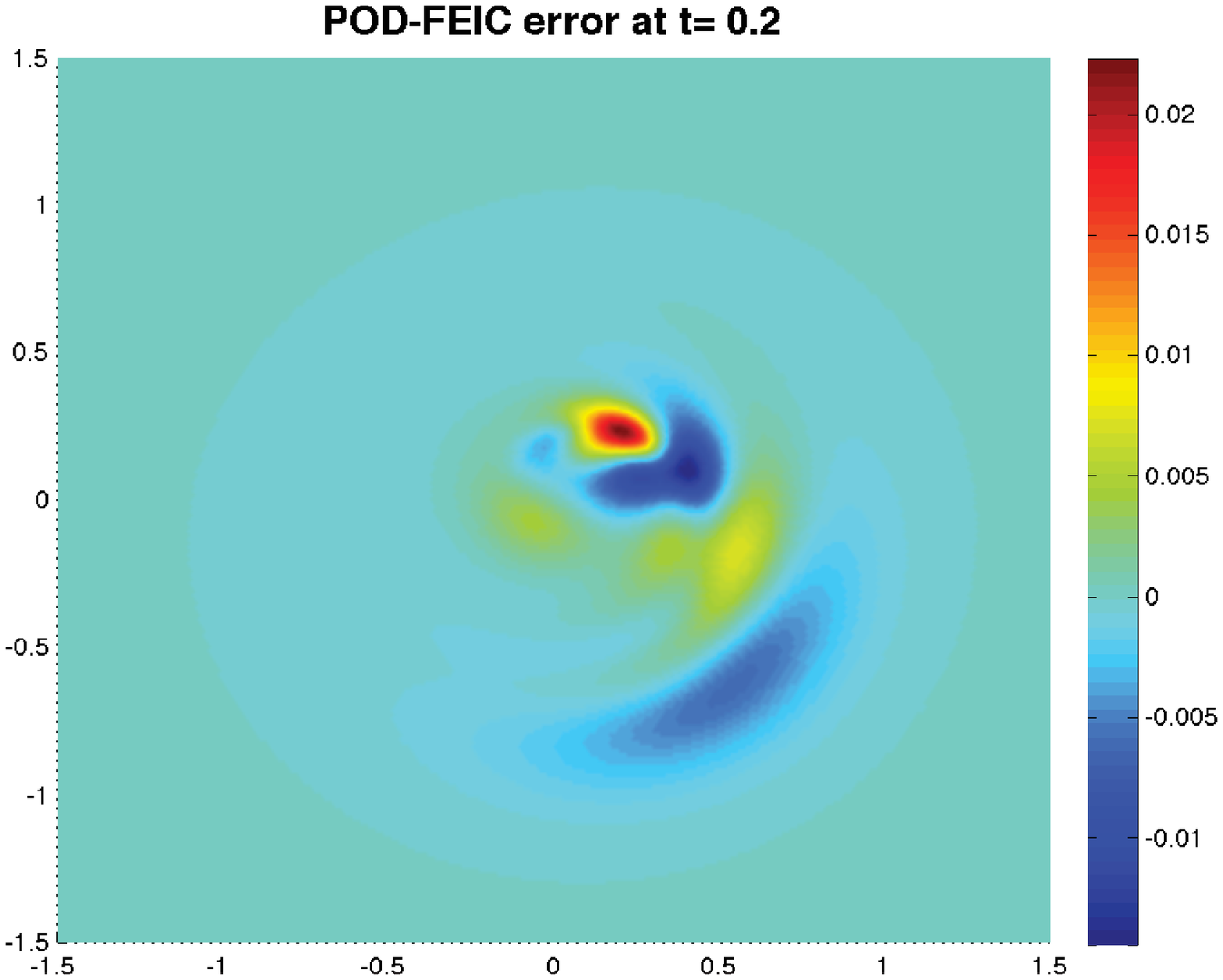}
\end{minipage}
\caption{
The simulations at $t=0.2$: the FEM solution (left), the error of POD-FEM with $r =5$ (middle), and the error of POD-FEIC with $r=5$ (right).  
}\label{Fig:bl_r}
\end{figure} 

\subsection{Verification}
In this subsection, we will verify the theoretical results obtained in Section \ref{sec:analysis} through a problem with a non-polynomial nonlinearity. 
It is governed by the following equations: 
\begin{eqnarray}
\frac{\p u}{\p t} - \Delta u + \sin(u) = f , &\quad x\in\Om, t\in [0, T], \label{eq:sin_1} \\
u = 0, &\quad x\in \partial \Om, t\in [0, T], \nonumber \\
u(x, 0) = g, &\quad x\in \Om, \nonumber
\label{eq:sin}
\end{eqnarray}  
where $f$ is determined by substituting a designated exact solution into the LHS of  \eqref{eq:sin_1} and $g$ is the exact solution at $t=0$. 
In the test, we consider the problem in 1D with exact solution $u= 0.5\sin(\pi x)(10\tanh(x-t)+1)$ on the domain $\Om = [0 ,1]$ during the time interval $t\in [0, 1]$. 
The exact solution is also our benchmark when calculating errors.
The weak formulation is of the form \eqref{eq:pdewf1} with $a(u, v) = (\nabla u, \nabla v)$ and $N(u) = \sin(u)$. 
We investigate the convergence properties of the POD-FEIC solution w.r.t. mesh size $h$ and the number of POD basis $r$, respectively. 

To check the approximation order of the POD-FEIC solution w.r.t. $h$, we collect the finite element solution of the original system with linear elements (m = 1) and quadratic elements (m=2) respectively. 
Backward-Euler method is used for the time integration with a small time step $\Delta t = 1\times 10^{-6}$. 
The number of POD basis functions are chosen such that $\sum_{j> r}\lambda_j <1\times 10^{-7}$. 
In this way, the spatial discretization error dominates the whole approximation property.  

The errors in both $\mathcal{E}_0$ and $\mathcal{E}_1$ norms are shown in Table \ref{table:sin1d_h}. 
Linear regressions indicate that, for linear elements, the order of convergence is 1.97 in $\mathcal{E}_0$ norm and 0.98 in $\mathcal{E}_1$ norm; 
for quadratic elements, the error convergence order is $2.94$ in $\mathcal{E}_0$ norm and 1.95 in $\mathcal{E}_1$ norm. 
The approximation orders are close to the optimal values in Theorem \ref{thm:err}. 
Since the error analysis is asymptotic, as $h$ decreases, the order approaches the optimal values (order $m+1$ in $\mathcal{E}_0$ norm and order $m$ in $\mathcal{E}_1$ norm). 

\begin{table}[htb]
\centering
\caption{The approximation rate of the POD-FEIC w.r.t. $h$}\label{table:sin1d_h}
\centering
\begin{tabular}{c|cc|cc}
\hline
\multirow{2}{*}{$h$}& \multicolumn{2}{c}{$m = 1$}&\multicolumn{2}{|c}{$m = 2$}\\
\cline{2-5}
{}&$\mathcal{E}_0(\urh, u)$ &$\mathcal{E}_1(\urh, u)$ & $\mathcal{E}_0(\urh, u)$ & $\mathcal{E}_1(\urh, u)$ \\
\hline
 1/8   & 2.09e-02 & 5.41e-01 & 2.41e-3 & 1.26e-1 \\
 1/16 & 5.50e-03 & 2.81e-01 & 3.30e-4 & 3.42e-2 \\ 
 1/32 & 1.39e-03 & 1.42e-01 & 4.22e-5 & 8.76e-3 \\ 
 1/64 & 3.50e-04 & 7.12e-02 & 5.31e-6 & 2.20e-3 \\
 \hline
 order& $h^{1.97}$ & $h^{0.98}$ & $h^{2.94}$ & $h^{1.95}$ \\
\hline
\end{tabular}
\end{table}

To check the approximation order of the POD-FEIC solution w.r.t. $r$, we collect the finite element solution of the original system when linear elements (p = 1) and quadratic elements (p=2) are utilized for spatial discretization, respectively, and backward-Euler method for the time integration. 
The mesh size $h=1/64$ and the time step $\Delta t = 1\times 10^{-6}$ are fixed.
The number of POD basis functions are chosen so that $\sqrt{\sum_{j> r}\lambda_j}$ decays by a factor of $2$.  

The errors in both $\mathcal{E}_0$ and $\mathcal{E}_1$ norms when linear elements are used are shown in Table \ref{table:sin1d_r_p1} and those for quadratic elements are listed in Table \ref{table:sin1d_r_p2}. 
Linear regressions indicate that the convergence order of error in $\mathcal{E}_1$ norm w.r.t. $\sqrt{\sum_{j> r}\|\varphi_j\|_{H^1}^2\,  \lambda_j}$ is 0.96 for linear elements and $1.00$ for quadratic elements. 
Wherein, the approximation order for quadratic elements is linear,  which is optimal as indicated in Theorem \ref{thm:err}. 
That of linear elements is close to 1 but slightly smaller than 1, which is because the discretization error tends to dominate the error as $r$ increases. 
At the same time, note that the approximation order of error in $\mathcal{E}_0$ norm w.r.t. $\sqrt{\sum_{j> r}\lambda_j}$ is 1 for both linear and quadratic elements. 
Although not proven theoretically, the error in $\mathcal{E}_0$ norm also converges optimally. 

\begin{table}[htb]
\centering
\caption{The approximation rate of the POD-FEIC w.r.t. $r$ when linear finite elements are used.}\label{table:sin1d_r_p1}
\centering
\begin{tabular}{c|cc|cc}
\hline
\multirow{2}{*}{$r$}& \multicolumn{4}{c}{m = 1}\\ 
\cline{2-5}
{}& $\sqrt{\sum_{j> r}\lambda_j}$ & $\mathcal{E}_0(\urh, u)$ & $\sqrt{ \sum_{j> r}\|\varphi_j\|_{H^1}^2 \lambda_j}$  & $\mathcal{E}_1(\urh, u)$ \\
\hline
 3 & 4.13e-02 & 4.60e-02 & 6.26e-01 & 6.21e-01  \\  
 4 & 2.45e-02 & 2.74e-02 & 4.48e-01 & 4.47e-01  \\  
 6 & 9.07e-03 & 1.02e-02 & 2.23e-01 & 2.31e-01  \\  
 7 & 5.59e-03 & 6.29e-03 & 1.55e-01 &1.69e-01   \\  
 \hline
 order& - & $1.00$ & - & $0.96$ \\
\hline
\end{tabular}
\end{table} 

\begin{table}[htb]
\centering
\caption{The approximation rate of the POD-FEIC w.r.t. $r$ when quadratic finite elements are used.}\label{table:sin1d_r_p2}
\centering
\begin{tabular}{c|cc|cc}
\hline
\multirow{2}{*}{$r$}& \multicolumn{4}{c}{m = 2}\\
\cline{2-5}
{}& $\sqrt{\sum_{j> r}\lambda_j}$ & $\mathcal{E}_0(\urh, u)$ & $\sqrt{ \sum_{j> r}\|\varphi_j\|_{H^1}^2 \lambda_j}$  & $\mathcal{E}_1(\urh, u)$ \\
\hline
 3 & 4.15e-02 & 4.60e-02 & 6.28e-01 & 6.19e-01  \\
 4 & 2.46e-02 & 2.74e-02 & 4.50e-01 & 4.43e-01  \\
 6 & 9.18e-03 & 1.02e-02 & 2.25e-01 & 2.21e-01  \\
 7 & 5.68e-03 & 6.29e-03 & 1.57e-01 & 1.54e-01  \\
 9 & 2.20e-03 & 2.41e-03 & 7.44e-02 & 7.36e-02  \\
 \hline
 order& - & $1.00$ & - & $1.00$ \\
\hline
\end{tabular}
\end{table} 

%% file: FEIC_DEIM.tex
\section{The Combination with the DEIM}\label{sec:FEIC-DEIM}
The discrete empirical interpolation method has been successfully applied in many nonlinear ROMs to reduce the computation complexity of the nonlinear terms \cite{chaturantabut2012state,chaturantabut2010nonlinear,chaturantabut2011application,hinze2012discrete,kellems2010morphologically,cstefuanescu2012pod}. 
For a general nonlinear function $N(u(x, t))$, the DEIM employs the ansatz 
\begin{equation}
N(u) = \sum\limits_{j=1}^{\widehat{r}} \psi_j(x) c_j(t), 
\end{equation}
where $\psi_j(x)$ is the $j$-th nonlinear POD basis obtained by applying the POD method on the nonlinear snapshots and 
$\widehat{r}$ is the rank of the nonlinear POD basis. 
Based on the nonlinear POD basis vectors ${\bf \Psi} = [\psi _1, \ldots, \psi_{\widehat{r}}]$, 
the DEIM optimally selects a set of interpolation points $\wp := [\wp_1, \ldots, \wp_p]$ and approximates the nonlinear function by 
\begin{equation}
{N}(u) \approx {\bf \Psi}(\bP^\intercal {\bf \Psi})^{-1} \bP^\intercal {\bf N}(u),
\label{deim}
\end{equation} 
where $\bP$ is the matrix for selecting the corresponding $p$ indices $\wp_1, \ldots, \wp_p$. 

However, when the FE method is used for a spatial discretization, the nonlinear snapshot becomes $(N(u(x, t)), h)$ as in the weak formulation \eqref{eq:fem1}. 
Therefore, generating nonlinear snapshots requires the inner product to be fulfilled, which costs lots of off-line time. 
In cases such as complex flows are studied, many interpolation points might be necessary to obtain a good approximation of the nonlinear terms. 
Since on-line simulations also need to evaluate the inner product over the elements sharing the selected DEIM points,  the on-line time increases.  
These issues represent the main computational hurdles for applying DEIM in the FE setting, which, however, can be easily overcome by the POD-FEIC approach we proposed in Section \ref{sec:feic}.  
Indeed, in the FEIC, the nonlinear snapshot can be chosen to be the value vector of the nonlinear function, $N({\bf Q}\, {\bf a}(t))$, thus, no any evaluation of inner product needed. 
Replacing the nonlinear function with the DEIM approximation \eqref{deim} in the POD-FEIC \eqref{eq:podgfeic}, we get the POD-FEIC-DEIM model, which has a more efficient approximation of the nonlinear term 
\begin{equation}
\mN_{FEIC-DEIM}= {\bf Q}^\intercal\, {\bf M}^h\,  {\bf \Psi}(\bP^\intercal {\bf \Psi})^{-1} \bP^\intercal {\bf N}\big({\bf Q}\, {\bf a}(t)\big).
\label{feic_deim}
\end{equation}

It is seen that, in on-line simulations, one only need to calculate the nonlinear functions at $p$ selected DEIM points, which doesn't involve any numerical quadratures. Therefore, the POD-FEIC-DEIM improves the computational efficiency over the POD-FEIC, which also outperforms the POD-FEM.  
We will demonstrate the effectiveness of the new approach by considering the first two examples in Section \ref{sec:numeric} again. 

\paragraph{FitzHugh-Nagumo System} 
We revisit the 1D F-N model used in Section \ref{sec:validation}. 
Totally, 801 nonlinear snapshots of $N(v) = \frac{1}{\mu} v(v-0.1)(1-v)$ are generated and used in the DEIM. 
The nonlinear function is then approximated by the DEIM basing on $p$ selected interpolation points. 
In this test, we consider the POD-FEIC-DEIM generated by the first $r=5$ POD basis functions and investigate the numerical performance of the new model by varying the number of interpolation points. 

Errors in $\mathcal{E}_0$ norm and the CPU time for simulations are listed in Table \ref{table:fn_feic_deim}. 
As $p$ increases, the POD-FEIC-DEIM result approaches to that of the POD-FEIC 
($\mathcal{E}_0(\urh, u^h)= 6.60\times 10^{-4}$ when $r=5$). 
It is also seen from Table \ref{table:fn_feic_deim} that when $p=3$, the POD-FEIC-DEIM error is 9 times larger than that of the POD-FEIC. 
However, as the limit cycle on the $v-w$ plane shown in Figure \ref{Fig:fn_feic_deim}, 
the difference mainly occurs at the beginning of the simulation, from $t=0$ to $t=1$. 
After the transient interval, the limit cycles of the POD-FEIC and the POD-FEIC-DEIM coincide with each other.  
What's more, the CPU time of the POD-FEIC-DEIM is lower than that of the POD-FEIC, although not significantly due to the small size of the tested problem.  

\begin{table}[htb]
\centering
\caption{
Errors and the CPU time of the POD-FEIC-DEIM model with $r= 5$ POD basis and $p$ interpolation points for nonlinear function approximation.  
}\label{table:fn_feic_deim}
\centering
\begin{tabular}{ccc}
\hline
$p$&$\mathcal{E}_0(\urh, u^h)$ & CPU time  \\
\hline
 3 & 5.94e-03& 1.02  \\
 5 & 3.00e-03& 1.05   \\
 7 & 9.66e-04& 1.09   \\
 9 & 6.63e-04& 1.09   \\
\hline
\end{tabular}
\end{table} 
\begin{figure}[htb]
\centering
\includegraphics[width=0.5\textwidth]{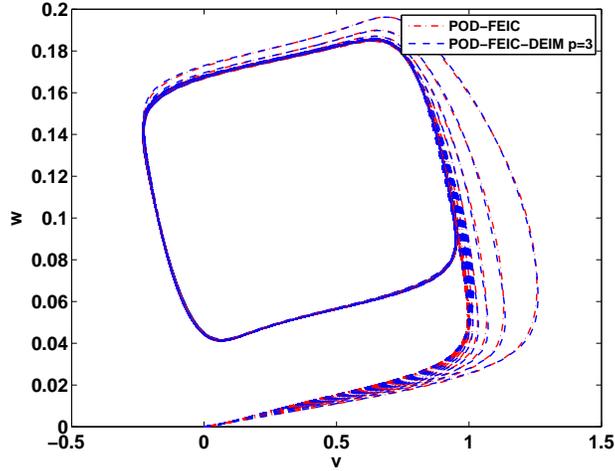}
\caption{
The limit cycle on the $v-w$ plane: the POD-FEIC results ($r=5$) and the POD-FEIC-DEIM approximation ($r=5$ and $p=3$).  
}\label{Fig:fn_feic_deim}
\end{figure} 

\paragraph{Buckley-Leverett Equation}  
We also consider the 2D BLE problem utilized in Section \ref{sec:validation}. 
Totally, 51 nonlinear snapshots of $f_1(u)$ and $f_2(u)$ are generated, and then used in the DEIM for selecting $p$ interpolation points, respectively. 
We also consider the POD-FEIC-DEIM model generated by the first $r=5$ POD basis functions and investigate its numerical behavior by varying $p$. 

Errors in $\mathcal{E}_0$ norm and the CPU time for simulations are listed in Table \ref{table:bl_feic_deim}. 
Note that for the POD-FEIC when $r=5$, the error $\mathcal{E}_0(\urh, u^h) = 4.43\times 10^{-3}$ and CPU time equals 116 $s$,  the POD-FEIC-DEIM model achieves a close accuracy by only using $p=5$ interpolation points, which also improves the computational efficiency of the POD-FEIC by 4 times. 
Figure \ref{Fig:bl_error} shows the distribution of  20 firstly selected DEIM interpolation points (left), and the difference between the POD-FEIC and the POD-FEIC-DEIM when $p = 5$ (middle) and $p=20$ (right).

\begin{table}[htb]
\centering
\caption{
Errors and the CPU time of the POD-FEIC-DEIM model with $r= 5$ POD basis and $p$ interpolation points.  
}\label{table:bl_feic_deim}
\centering
\begin{tabular}{ccc}
\hline
$p$&$\mathcal{E}_0(\urh, u^h)$ & CPU time  \\
\hline
 5  & 4.81e-03 & 27.79  \\
10 & 4.64e-03 & 27.89   \\
15 & 4.59e-03 & 28.00   \\
20 & 4.54e-03 & 28.57   \\
\hline
\end{tabular}
\end{table}

\begin{figure}[htb]
\centering
\begin{minipage}[ht]{0.27\linewidth}
\includegraphics[width=1\textwidth]{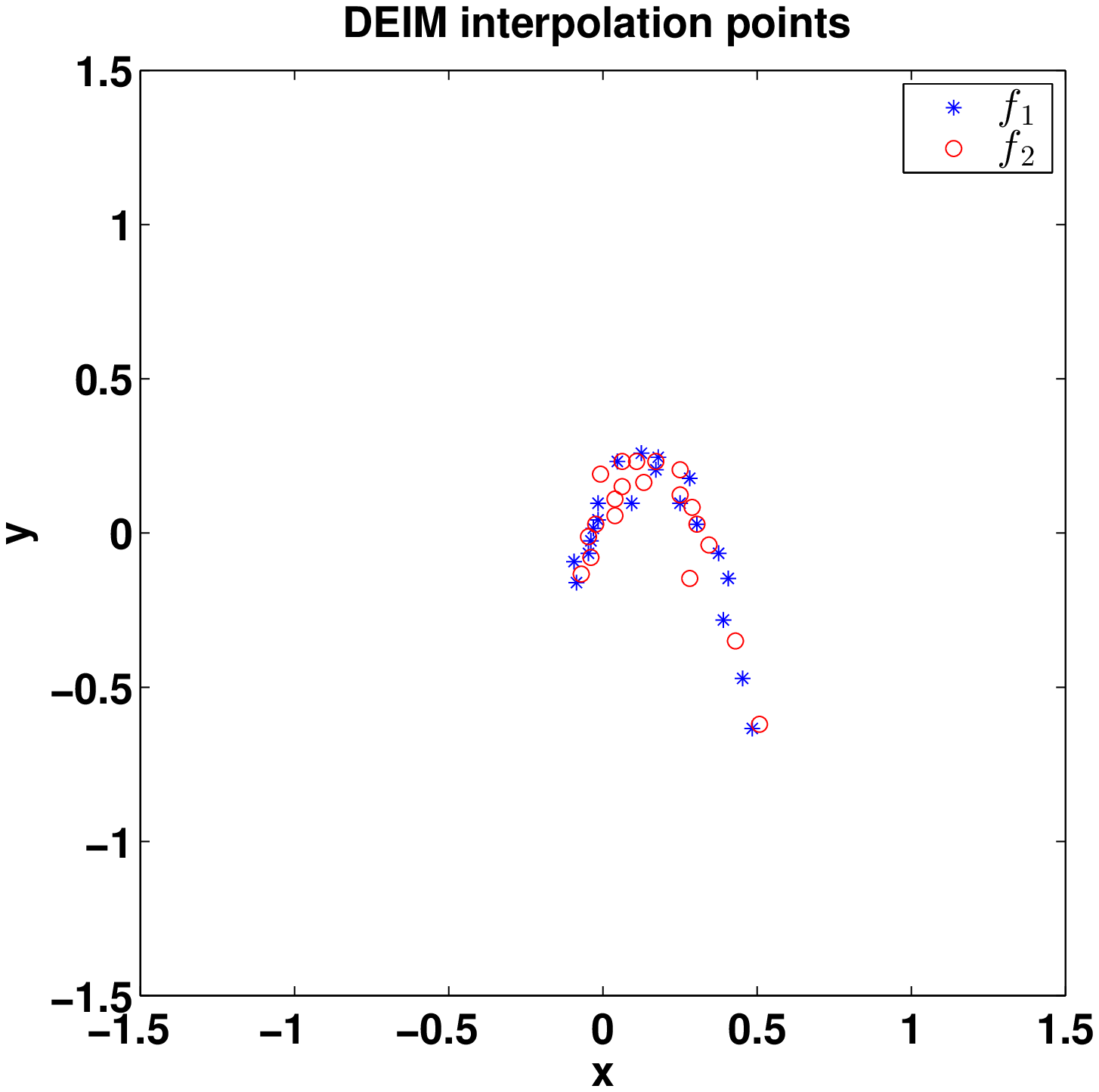}
\end{minipage}
\begin{minipage}[ht]{0.33\linewidth}
\includegraphics[width=1\textwidth]{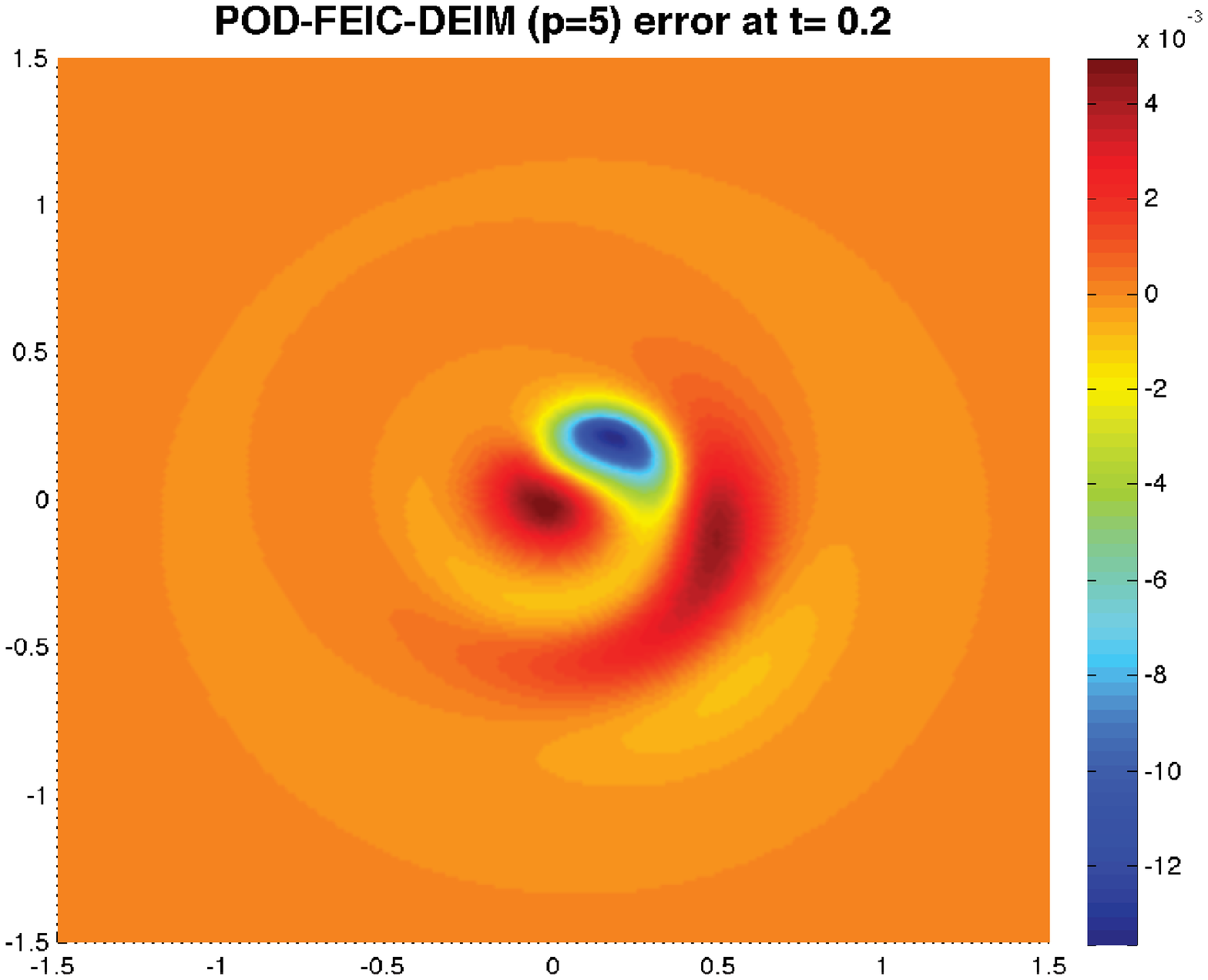}
\end{minipage}
\begin{minipage}[ht]{0.33\linewidth}
\includegraphics[width=1\textwidth]{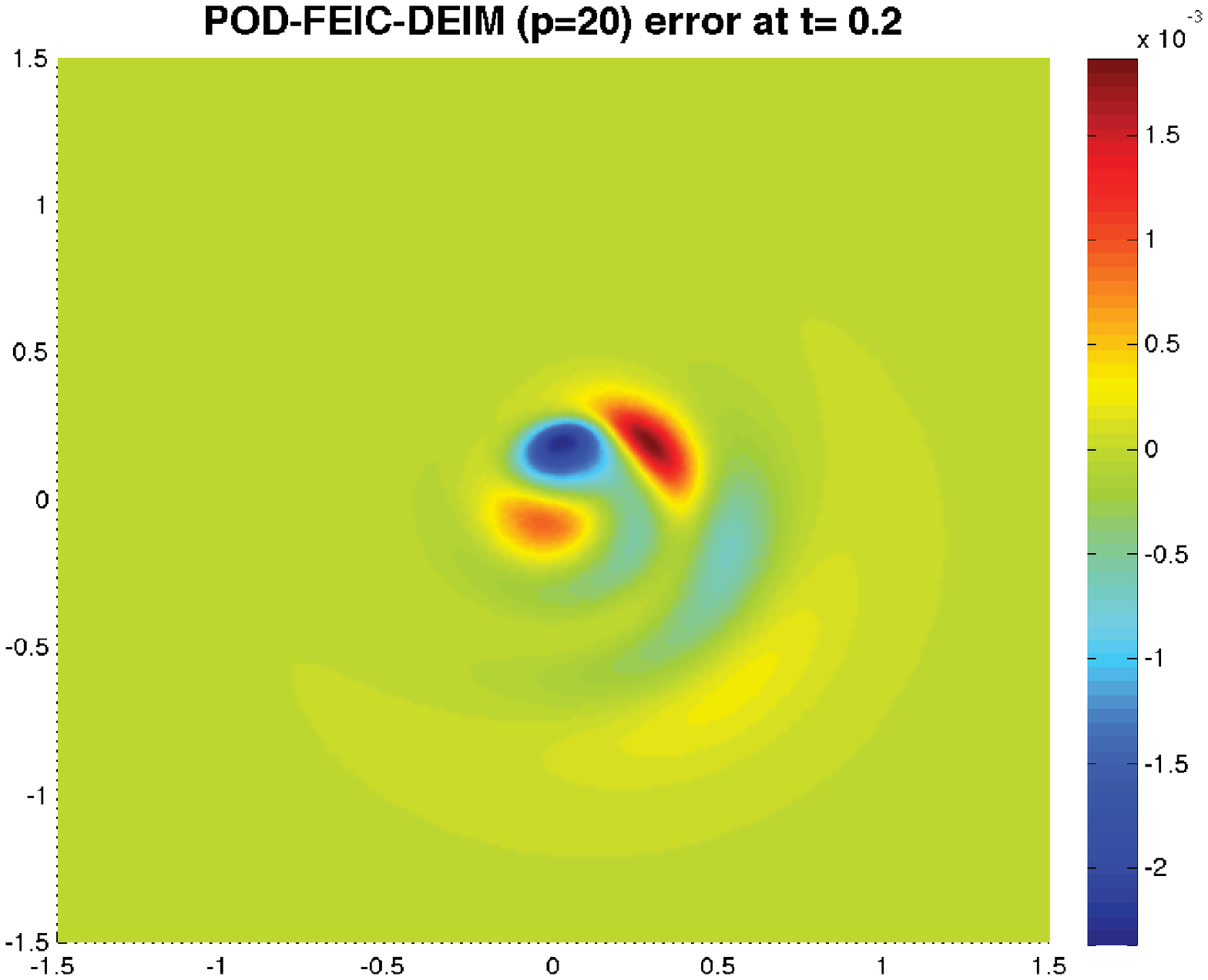}
\end{minipage}
\caption{
Distribution of DEIM interpolation points for $f_1(u)$ and $f_2(u)$ (left), the difference between POD-FEIC and POD-FEIC-DEIM at $t=0.2$ when $p=5$ (middle) and $p=20$ (right).  
}\label{Fig:bl_error}
\end{figure}

%% file: Conclusion.tex
\section{Conclusions}
As a first step of our investigations on efficient finite element discretization algorithms for nonlinear model reduction techniques, we develop the finite element method with interpolated coefficients for nonlinear POD-ROMs. 
Comparing with the standard finite element discretization, the proposed approach is computationally more efficient because the coefficient matrices in the nonlinear terms can be precomputed and there is no any numerical quadratures needed during the simulation.   
The proposed method also achieves the same accuracy as that of the standard finite element discretization when nonlinear functions satisfy certain smoothness assumptions. 
The proposed approach is a more suitable base for the discrete empirical interpolation method. 
Combining the FEIC with the DEIM will further reduce the computational complexity for evaluating nonlinear terms. 


We plan to continue investigating several research avenues:  
we will first extend the proposed approach to nonlinear closure models we developed for complex flows in \cite{wang2011two,wang2011closure}. Some realistic engineering application problems will be tested.  
Second, we will develop a rigorous error analysis of the model reduction approach that combining the FEIC and the DEIM. 
Finally, we plan to design an adaptive algorithm for the proposed method, which will be used to supervise the approximation error control. 

%% file: Main_NonlinearPOD.bbl
\begin{thebibliography}{10}

\bibitem{amsallem2012nonlinear}
D.~Amsallem, M.~J. Zahr, and C.~Farhat.
\newblock Nonlinear model order reduction based on local reduced-order bases.
\newblock {\em International Journal for Numerical Methods in Engineering},
  92(10):891--916, 2012.

\bibitem{astrid2004reduction}
P.~Astrid.
\newblock {\em Reduction of process simulation models: a proper orthogonal
  decomposition approach}.
\newblock PhD thesis, Technische Universiteit Eindhoven, 2004.

\bibitem{astrid2008missing}
P.~Astrid, S.~Weiland, K.~Willcox, and T.~Backx.
\newblock Missing point estimation in models described by proper orthogonal
  decomposition.
\newblock {\em Automatic Control, IEEE Transactions on}, 53(10):2237--2251,
  2008.

\bibitem{aubry1993preserving}
N.~Aubry, W.~Y. Lian, and E.~S. Titi.
\newblock {Preserving symmetries in the proper orthogonal decomposition}.
\newblock {\em SIAM Journal on Scientific Computing}, 14:483--505, 1993.

\bibitem{barrault2004empirical}
M.~Barrault, Y.~Maday, N.~C. Nguyen, and A.~T. Patera.
\newblock An {'}empirical interpolation{'} method: application to efficient
  reduced-basis discretization of partial differential equations.
\newblock {\em Comptes Rendus Mathematique}, 339(9):667--672, 2004.

\bibitem{bergmann2009enablers}
M.~Bergmann, C.~H. Bruneau, and A.~Iollo.
\newblock {Enablers for robust POD models}.
\newblock {\em Journal of Computational Physics}, 228(2):516--538, 2009.

\bibitem{bjorklund2012hierarchical}
M.~Bj{\"o}rklund.
\newblock {\em A hierarchical POD reduction method of finite element models
  with application to simulated mechanical systems}.
\newblock PhD thesis, Ume{\aa} University, 2012.

\bibitem{borggaard2008reduced}
J.~Borggaard, A.~Duggleby, A.~Hay, T.~Iliescu, and Z.~Wang.
\newblock Reduced-order modeling of turbulent flows.
\newblock In {\em Proceedings of MTNS 2008}, 2008.

\bibitem{borggaard2006interval}
J.~Borggaard, A.~Hay, and D.~Pelletier.
\newblock Interval-based reduced order models for unsteady fluid flow.
\newblock {\em International Journal of Numerical Analysis \& Modeling},
  4(3--4):353--367, 2007.

\bibitem{borggaard2011artificial}
J.~Borggaard, T.~Iliescu, and Z.~Wang.
\newblock Artificial viscosity proper orthogonal decomposition.
\newblock {\em Mathematical and Computer Modelling}, 53(1-2):269--279, 2011.

\bibitem{BS02}
S.~C. Brenner and L.~R. Scott.
\newblock {\em The mathematical theory of finite element methods}, volume~15 of
  {\em Texts in Applied Mathematics}.
\newblock Springer-Verlag, New York, second edition, 2002.

\bibitem{bui2007goal}
T.~Bui-Thanh, K.~Willcox, O.~Ghattas, and B.~van Bloemen~Waanders.
\newblock Goal-oriented, model-constrained optimization for reduction of
  large-scale systems.
\newblock {\em Journal of Computational Physics}, 224(2):880--896, 2007.

\bibitem{burkardt2006pod}
J.~Burkardt, M.~Gunzburger, and H.~C. Lee.
\newblock {POD and CVT-based reduced-order modeling of Navier--Stokes flows}.
\newblock {\em Computer Methods in Applied Mechanics and Engineering},
  196(1-3):337--355, 2006.

\bibitem{carlberg2011low}
K.~Carlberg and C.~Farhat.
\newblock A low-cost, goal-oriented Ôcompact proper orthogonal
  decompositionÕbasis for model reduction of static systems.
\newblock {\em International Journal for Numerical Methods in Engineering},
  86(3):381--402, 2011.

\bibitem{chapelle2012galerkin}
D.~Chapelle, A.~Gariah, and J.~Sainte-Marie.
\newblock Galerkin approximation with proper orthogonal decomposition: new
  error estimates and illustrative examples.
\newblock {\em ESAIM: Mathematical Modelling and Numerical Analysis},
  46(04):731--757, 2012.

\bibitem{chaturantabut2011application}
S.~Chaturantabut and D.~C. Sorensen.
\newblock Application of {POD} and {DEIM} on dimension reduction of non-linear
  miscible viscous fingering in porous media.
\newblock {\em Mathematical and Computer Modelling of Dynamical Systems},
  17(4):337--353, 2011.

\bibitem{chaturantabut2012state}
S.~Chaturantabut and D.~C. Sorensen.
\newblock A state space error estimate for {POD-DEIM} nonlinear model
  reduction.
\newblock {\em SIAM Journal on Numerical Analysis}, 50:46--63, 2012.

\bibitem{chaturantabut2010nonlinear}
S.~Chaturantabut, D.~C. Sorensen, and J.~C. Steven.
\newblock Nonlinear model reduction via discrete empirical interpolation.
\newblock {\em SIAM Journal on Scientific Computing}, 32(5):2737--2764, 2010.

\bibitem{chen1989error}
C.~M. Chen, S.~Larsson, and N.~Y. Zhang.
\newblock Error estimates of optimal order for finite element methods with
  interpolated coefficients for the nonlinear heat equation.
\newblock {\em IMA Journal of Numerical Analysis}, 9(4):507--524, 1989.

\bibitem{chen1991approximation}
Z.~X. Chen and Jr. J.~Douglas.
\newblock Approximation of coefficients in hybrid and mixed methods for
  nonlinear parabolic problems.
\newblock {\em Matem\'atica Aplicada E Computacional}, 10:137--160, 1991.

\bibitem{christie1981product}
I.~Christie, D.~F. Griffiths, A.~R. Mitchell, and J.~M. Sanz-Serna.
\newblock Product approximation for non-linear problems in the finite element
  method.
\newblock {\em IMA Journal of Numerical Analysis}, 1(3):253--266, 1981.

\bibitem{cordier2010calibration}
L.~Cordier, E.~Majd, B.~Abou, and J.~Favier.
\newblock Calibration of {POD} reduced-order models using tikhonov
  regularization.
\newblock {\em International Journal for Numerical Methods in Fluids},
  63(2):269--296, 2010.

\bibitem{couplet2005cro}
M.~Couplet, C.~Basdevant, and P.~Sagaut.
\newblock Calibrated reduced-order {POD}-{G}alerkin system for fluid flow
  modelling.
\newblock {\em Journal of Computational Physics}, 207:192--220, 2005.

\bibitem{daescu2008dual}
D.~N. Daescu and I.~M. Navon.
\newblock A dual-weighted approach to order reduction in {4DVAR} data
  assimilation.
\newblock {\em Monthly Weather Review}, 136(3):1026--1041, 2008.

\bibitem{dickinson2010nmr}
B.~T. Dickinson and J.~R. Singler.
\newblock Nonlinear model reduction using group proper orthogonal
  decomposition.
\newblock {\em International Journal of Numerical Analysis \& Modeling},
  7(2):356--372, 2010.

\bibitem{eftang2011parameter}
J.~L Eftang and B.~Stamm.
\newblock Parameter multi-domain ÔhpÕ empirical interpolation.
\newblock {\em International Journal for Numerical Methods in Engineering},
  90(4):412--428, 2012.

\bibitem{fang2013nonlinear}
F.~Fang, C.~C. Pain, I.~M. Navon, A.~H. Elsheikh, J.~Du, and D.~Xiao.
\newblock Non-linear {P}etrov-{G}alerkin methods for reduced order hyperbolic
  equations and discontinuous finite element methods.
\newblock {\em Journal of Computational Physics}, 234:540--559, 2013.

\bibitem{fletcher1983group}
C.~A.~J. Fletcher.
\newblock The group finite element formulation.
\newblock {\em Computer Methods in Applied Mechanics and Engineering},
  37(2):225--244, 1983.

\bibitem{galletti2007accurate}
B.~Galletti, A.~Bottaro, C.~H. Bruneau, and A.~Iollo.
\newblock Accurate model reduction of transient and forced wakes.
\newblock {\em European Journal of Mechanics-B/Fluids}, 26(3):354--366, 2007.

\bibitem{Galletti2004}
B.~Galletti, C.~H. Bruneau, L.~Zannetti, and A.~Iollo.
\newblock Low-order modeling of laminar flow regimes past a confined square
  cylinder.
\newblock {\em Journal of Fluid Mechanics}, 503:161--170, 2004.

\bibitem{grepl2007erb}
M.~A. Grepl, Y.~Maday, N.~C. Nguyen, and A.~T. Patera.
\newblock Efficient reduced-basis treatment of nonaffine and nonlinear partial
  differential equations.
\newblock {\em ESAIM: Mathematical Modelling and Numerical Analysis},
  41(3):575--605, 2007.

\bibitem{Gun03}
M.~Gunzburger.
\newblock {\em Perspectives in flow control and optimization}.
\newblock Advances in Design and Control. Society for Industrial and Applied
  Mathematics (SIAM), Philadelphia, PA, 2003.

\bibitem{hesthaven2011efficient}
J.~S. Hesthaven, B.~Stamm, and S.~Zhang.
\newblock Efficient greedy algorithms for high-dimensional parameter spaces
  with applications to empirical interpolation and reduced basis methods.
\newblock Technical report, DTIC Document, 2011.

\bibitem{hinze2012discrete}
M.~Hinze and M.~Kunkel.
\newblock Discrete empirical interpolation in {POD} model order reduction of
  drift-diffusion equations in electrical networks.
\newblock In {\em Scientific Computing in Electrical Engineering SCEE 2010},
  pages 423--431. Springer, 2012.

\bibitem{HLB12}
P.~Holmes, J.~L. Lumley, G.~Berkooz, and C.~W. Rowley.
\newblock {\em Turbulence, coherent structures, dynamical systems and
  symmetry}.
\newblock Cambridge University Press, Cambridge, UK, 2012.

\bibitem{wang2011variational}
T.~Iliescu and Z.~Wang.
\newblock Variational multiscale proper orthogonal decomposition:
  Convection-dominated convection-diffusion-reaction equations.
\newblock {\em Mathematics of Computation}, 2011.
\newblock In press.

\bibitem{iollo2000two}
A.~Iollo, A.~Dervieux, J.~A. D{\'e}sid{\'e}ri, and S.~Lanteri.
\newblock {Two stable POD-based approximations to the Navier--Stokes
  equations}.
\newblock {\em Computing and Visualization in Science}, 3(1):61--66, 2000.

\bibitem{iollo2000stability}
A.~Iollo, S.~Lanteri, and J.~A. D{\'e}sid{\'e}ri.
\newblock {Stability properties of POD-{G}alerkin approximations for the
  compressible {N}avier--{S}tokes equations}.
\newblock {\em Theoretical and Computational Fluid Dynamics}, 13(6):377--396,
  2000.

\bibitem{douglas1975effect}
Jr. J.~Douglas and T.~Dupont.
\newblock The effect of interpolating the coefficients in nonlinear parabolic
  galerkin procedures.
\newblock {\em Mathematics of Computation}, 29:360--389, 1975.

\bibitem{kalb2007intrinsic}
V.~L. Kalb and A.~E. Deane.
\newblock An intrinsic stabilization scheme for proper orthogonal decomposition
  based low-dimensional models.
\newblock {\em Physics of Fluids}, 19:054106, 2007.

\bibitem{kellems2010morphologically}
A.~R. Kellems, S.~Chaturantabut, D.~C. Sorensen, and S.~J. Cox.
\newblock Morphologically accurate reduced order modeling of spiking neurons.
\newblock {\em Journal of Computational Neuroscience}, 28(3):477--494, 2010.

\bibitem{KV99}
K.~Kunisch and S.~Volkwein.
\newblock Control of the {B}urgers equation by a reduced-order approach using
  proper orthogonal decomposition.
\newblock {\em Journal of Optimization Theory and Applications},
  102(2):345--371, 1999.

\bibitem{KV01}
K.~Kunisch and S.~Volkwein.
\newblock {G}alerkin proper orthogonal decomposition methods for parabolic
  problems.
\newblock {\em Numerische Mathematik}, 90(1):117--148, 2001.

\bibitem{KV02}
K.~Kunisch and S.~Volkwein.
\newblock {G}alerkin proper orthogonal decomposition methods for a general
  equation in fluid dynamics.
\newblock {\em SIAM Journal on Numerical Analysis}, 40(2):492--515
  (electronic), 2002.

\bibitem{kunisch2010optimal}
K.~Kunisch and S.~Volkwein.
\newblock Optimal snapshot location for computing {POD} basis functions.
\newblock {\em ESAIM: Mathematical Modelling and Numerical Analysis},
  44(3):509, 2010.

\bibitem{larsson1989interpolation}
S.~Larsson, V.~Thom{\'e}e, and N.~Y. Zhang.
\newblock Interpolation of coefficients and transformation of the dependent
  variable in finite element methods for the non-linear heat equation.
\newblock {\em Mathematical Methods in the Applied Sciences}, 11(1):105--124,
  1989.

\bibitem{lopez1988stability}
J.~C. L{\'o}pez-Marcos and J.~M. Sanz-Serna.
\newblock Stability and convergence in numerical analysis {III}: Linear
  investigation of nonlinear stability.
\newblock {\em IMA Journal of Numerical Analysis}, 8(1):71--84, 1988.

\bibitem{lumley1996dynamical}
J.~L. Lumley and B.~Podvin.
\newblock Dynamical systems theory and extra rates of strain in turbulent
  flows.
\newblock {\em Experimental Thermal and Fluid Science}, 13(3):180--189, 1996.

\bibitem{nguyen2008best}
N.~C. Nguyen, A.~T. Patera, and J.~Peraire.
\newblock A Ôbest pointsÕ interpolation method for efficient approximation of
  parametrized functions.
\newblock {\em International Journal for Numerical Methods in Engineering},
  73(4):521--543, 2008.

\bibitem{nguyen2008efficient}
N.~C. Nguyen and J.~Peraire.
\newblock An efficient reduced-order modeling approach for non-linear
  parametrized partial differential equations.
\newblock {\em International Journal for Numerical Methods in Engineering},
  76(1):27--55, 2008.

\bibitem{Noack2003}
B.~R. Noack, K.~Afanasiev, M.~Morzy{\'n}ski, and F.~Thiele.
\newblock A hierarchy of low-dimensional models for the transient and
  post-transient cylinder wake.
\newblock {\em Journal of Fluid Mechanics}, 497:335--363, 2003.

\bibitem{noack2011reduced}
B.~R. Noack, M.~Morzy{\'n}ski, and G.~Tadmor.
\newblock {\em Reduced-Order Modelling for Flow Control}, volume 528.
\newblock Springer Verlag, 2011.

\bibitem{noack2002low}
B.~R. Noack, P.~Papas, and P.~A. Monkewitz.
\newblock Low-dimensional {G}alerkin model of a laminar shear-layer.
\newblock Technical Report 2002-01, {\'E}cole Polytechnique F{\'e}d{\'e}rale de
  Lausanne, 2002.

\bibitem{noack2005}
B.~R. Noack, P.~Papas, and P.~A. Monkewitz.
\newblock The need for a pressure-term representation in empirical {G}alerkin
  models of incompressible shear flows.
\newblock {\em Journal of Fluid Mechanics}, 523:339--365, 2005.

\bibitem{Pod01}
B.~Podvin.
\newblock On the adequacy of the ten-dimensional model for the wall layer.
\newblock {\em Physics of Fluids}, 13(1):210--224, 2001.

\bibitem{podvin2009proper}
B.~Podvin.
\newblock A proper-orthogonal-decomposition--based model for the wall layer of
  a turbulent channel flow.
\newblock {\em Physics of Fluids}, 21:015111, 2009.

\bibitem{PL98}
B.~Podvin and J.~L. Lumley.
\newblock A low-dimensional approach for the minimal flow unit.
\newblock {\em Journal of Fluid Mechanics}, 362:121--155, 1998.

\bibitem{rempfer1996investigations}
D.~Rempfer.
\newblock {Investigations of boundary layer transition via {G}alerkin
  projections on empirical eigenfunctions}.
\newblock {\em Physics of Fluids}, 8:175, 1996.

\bibitem{rempfer2000low}
D.~Rempfer.
\newblock {On low-dimensional {G}alerkin models for fluid flow}.
\newblock {\em Theoretical and Computational Fluid Dynamics}, 14(2):75--88,
  2000.

\bibitem{RF94}
D.~Rempfer and H.~F. Fasel.
\newblock Dynamics of three-dimensional coherent structures in a flat-plate
  boundary layer.
\newblock {\em Journal of Fluid Mechanics}, 275:257--283, 1994.

\bibitem{rewienski2003trajectory}
M.~Rewie{\'n}ski and J.~White.
\newblock A trajectory piecewise-linear approach to model order reduction and
  fast simulation of nonlinear circuits and micromachined devices.
\newblock {\em IEEE Transactions on Computer-Aided Design of Integrated
  Circuits and Systems}, 22(2):155--170, 2003.

\bibitem{rewienski2006model}
M.~Rewie{\'n}ski and J.~White.
\newblock Model order reduction for nonlinear dynamical systems based on
  trajectory piecewise-linear approximations.
\newblock {\em Linear Algebra and its Applications}, 415(2):426--454, 2006.

\bibitem{sanz1984interpolation}
J.~M. Sanz-Serna and L.~Abia.
\newblock Interpolation of the coefficients in nonlinear elliptic {G}alerkin
  procedures.
\newblock {\em SIAM Journal on Numerical Analysis}, 21:77--83, 1984.

\bibitem{Singler2013}
J.~R. Singler.
\newblock New {POD} error expressions, error bounds, and asymptotic results for
  reduced order model of parabolic {PDEs}.
\newblock {\em preprint}.

\bibitem{sirisup2004spectral}
S.~Sirisup and G.~E. Karniadakis.
\newblock {A spectral viscosity method for correcting the long-term behavior of
  POD models}.
\newblock {\em Journal of Computational Physics}, 194(1):92--116, 2004.

\bibitem{Sir87a}
L.~Sirovich.
\newblock Turbulence and the dynamics of coherent structures. {I}. {C}oherent
  structures.
\newblock {\em Quarterly of Applied Mathematics}, 45(3):561--571, 1987.

\bibitem{cstefuanescu2012pod}
R.~Stef{\u{a}}nescu and I.~M. Navon.
\newblock {POD/DEIM} nonlinear model order reduction of an {ADI} implicit
  shallow water equations model.
\newblock {\em Journal of Computational Physics}, 237:95--114, 2013.

\bibitem{thomee2006galerkin}
V.~Thom{\'e}e.
\newblock {\em {G}alerkin finite element methods for parabolic problems}.
\newblock Springer Verlag, 2006.

\bibitem{tourigny1990product}
Y.~Tourigny.
\newblock Product approximation for nonlinear {Klein-Gordon} equations.
\newblock {\em IMA Journal of Numerical Analysis}, 10(3):449--462, 1990.

\bibitem{ullmann2010pod}
S.~Ullmann and J.~Lang.
\newblock A {POD-G}alerkin reduced model with updated coefficients for
  {S}magorinsky {LES}.
\newblock In J.~C.~F. Pereira and A.~Sequeira, editors, {\em V European
  Conference on Computational Fluid Dynamics, ECCOMAS CFD 2010}, Lisbon,
  Portugal, June 2010.

\bibitem{wang2011two}
Z.~Wang, I.~Akhtar, J.~Borggaard, and T.~Iliescu.
\newblock Two-level discretizations of nonlinear closure models for proper
  orthogonal decomposition.
\newblock {\em Journal of Computational Physics}, 230:126--146, 2011.

\bibitem{wang2011closure}
Z.~Wang, I.~Akhtar, J.~Borggaard, and T.~Iliescu.
\newblock Proper orthogonal decomposition closure models for turbulent flows: A
  numerical comparison.
\newblock {\em Computational Methods in Applied Mechanical Engineering},
  237--240:10--26, 2012.

\bibitem{weller2009feedback}
J.~Weller, S.~Camarri, and A.~Iollo.
\newblock Feedback control by low-order modelling of the laminar flow past a
  bluff body.
\newblock {\em Journal of Fluid Mechanics}, 634(1):405--418, 2009.

\bibitem{willcox2006unsteady}
K.~Willcox.
\newblock Unsteady flow sensing and estimation via the gappy proper orthogonal
  decomposition.
\newblock {\em Computers \& Fluids}, 35(2):208--226, 2006.

\bibitem{xiong2008convergence}
Z.~Xiong, Y.~Chen, and Y.~Zhang.
\newblock Convergence of {FEM} with interpolated coefficients for semilinear
  hyperbolic equation.
\newblock {\em Journal of Computational and Applied Mathematics},
  214(1):313--317, 2008.

\end{thebibliography}
